\documentclass{amsart}

\usepackage{amssymb}
\usepackage{MnSymbol}
\usepackage{graphicx}               
\usepackage{bbm}
\usepackage{mathtools}
\usepackage{extarrows}

\usepackage{color}

\newtheorem{theorem}{Theorem}[section]
\newtheorem{lemma}[theorem]{Lemma}
\newtheorem{proposition}[theorem]{Proposition}
\newtheorem{corollary}[theorem]{Corollary}
\newtheorem{Conditions}[theorem]{Conditions}

\theoremstyle{definition}
\newtheorem{definition}[theorem]{Definition}

\theoremstyle{remark}
\newtheorem{remark}[theorem]{Remark}

\numberwithin{equation}{section}

\newcommand{\matrixxx}[2]{\left(\begin{matrix}
#1\\
#2
\end{matrix} \right)}

\newcommand{\smallmatrixx}[4]{(\begin{smallmatrix}
#1 & #2\\
#3 & #4
\end{smallmatrix})}
\space

\newcommand{\matrixx}[4]{\left(\begin{matrix}
#1 & #2\\
#3 & #4
\end{matrix} \right)}
\space

\newcommand{\tensorspace}{\mathcal{O}_K\otimes_{\mathbb{Z}}\mathbb{Z}_p}

\newcommand{\pprime}{\mathfrak{p}}

\newcommand{\bonitaa}[4]{\scalebox{0.8}{%
$\matrixx{#1}{#2}{#3}{#4}$}}

\begin{document}

\title{Functional equation of the $p$-adic $L$-function of Bianchi modular forms}

\author{Luis Santiago Palacios}


\begin{abstract}
Let $K$ be an imaginary quadratic field with class number 1, in this paper we obtain the functional equation of the $p$-adic $L$-function of small slope $p$-stabilised Bianchi modular forms. Then, using $p$-adic families of Bianchi modular forms, we extend our result to $\Sigma$-smooth base-change Bianchi modular forms. 
\end{abstract}

\maketitle{}

\section{Introduction} 

Fix $p$ a rational prime. Via the theory of overconvergent modular symbols, Pollack and Stevens (see \cite{pollack2011} and \cite{pollack2013critical}, or for an exposition, \cite{pollack2014overconvergent}) gave a method of constructing the $p$-adic $L$-function of a suitable rational modular form.

In \cite{chris2017} was developed an analogue of the work of \cite{pollack2011} for the case of Bianchi modular forms, that is, the case of automorphic forms for $\mathrm{GL}_2$ over an imaginary quadratic field $K$. Let $\mathcal{F}\in S_{(k,k)}(\Gamma_0(\mathfrak{n}))$ be a cuspidal Bianchi modular form of weight $(k,k)$ and level $\mathfrak{n}$, if $\mathcal{F}$ is an eigenform of small slope and $(p)|\mathfrak{n}$, in \cite{chris2017} was constructed  $L_p(\mathcal{F},-)$, the $p$-adic $L$-function of $\mathcal{F}$ on $\mathfrak{X}(\mathrm{Cl}_K(p^{\infty}))$, the two-dimensional rigid space of $p$-adic characters on the ray class group $\mathrm{Cl}_K(p^{\infty})$.

In this paper, $K$ has class number 1.\footnote{We expect our results follow for higher class number taking direct sum over class group (as in \cite{chris2017}).}

Our first result is the functional equation of $L_p(\mathcal{F}_p,-)$ where $\mathcal{F}_p$ is a small slope Bianchi modular form obtained by successively stabilising at each different prime $\pprime$ above $p$ a newform $\mathcal{F}$. 

A functional equation for the classical $L$-function attached to a Bianchi modular form is given in \cite{JL}. In Section \ref{L-function}, we recover the following reformulation:

\begin{theorem}
Let $\mathcal{F}\in S_{(k,k)}(\Gamma_0(\mathfrak{n}))$ be a Bianchi newform and $\psi$ be a Hecke character of $K$ of conductor $\mathfrak{f}$ with $(\mathfrak{n},\mathfrak{f})=1$ and infinity type $0 \leqslant (q,r) \leqslant (k,k)$, we have 
\begin{equation*}
\Lambda(\mathcal{F},\psi)=\frac{-\epsilon(\mathfrak{n})|\nu|^k\tau(\psi|\cdot|_{\mathbb{A}_K}^{-k})}{\psi_\mathfrak{f}(-\nu)\psi_\infty(-\nu)\tau(\psi^{-1})}\Lambda(\mathcal{F},\psi^{-1}|\cdot|_{\mathbb{A}_K}^{k}),
\end{equation*}
where $\mathfrak{n}=(\nu)$, and $\epsilon(\mathfrak{n})=\pm1$ is the eigenvalue of $\mathcal{F}$ for the Fricke involution $W_\mathfrak{n}$.
\label{T.intro.1}
\end{theorem}
In theorem above, $\Lambda(\mathcal{F},\cdot)$ is the $L$-function of $\mathcal{F}$ renormalised by Deligne's $\Gamma$-factors at infinity, $\psi_\mathfrak{f}=\prod_{\mathfrak{q}|\mathfrak{f}}\psi_\mathfrak{q}$ with $\psi_{\mathfrak{q}}$ the restriction of $\psi$ to $K_{\mathfrak{q}}^\times$, $\psi_\infty$ is the infinite part of $\psi$, $\tau(\cdot)$ is the Gauss sum of \cite[\S1.2.3]{chris2017} and $|\cdot|_{\mathbb{A}_K}$ is the adelic norm.

We then obtain the functional equation of $L_p(\mathcal{F}_p,-)$ in the small slope case:

\begin{theorem}
Let $\mathcal{F}_p$ be a small slope $p$-stabilisation of a newform $\mathcal{F}\in S_{(k,k)}(\Gamma_0(\mathfrak{n}))$ with $\mathfrak{n}=(\nu)$ and $(p)\nmid\mathfrak{n}$. Then for all $\kappa\in\mathfrak{X}(\mathrm{Cl}_K(p^{\infty}))$, the distribution $L_p(\mathcal{F}_p,-)$ satisfies the following functional equation
\begin{equation*}
L_p(\mathcal{F}_p,\kappa)=-\epsilon(\mathfrak{n})\mathrm{N}(\mathfrak{n})^{k/2}\kappa(x_{-\nu,p})^{-1}L_p(\mathcal{F}_p,\kappa^{-1}\sigma_p^{k,k}),
\end{equation*}
where $\epsilon(\mathfrak{n})=\pm1$ is the eigenvalue of $\mathcal{F}$ for the Fricke involution $W_\mathfrak{n}$, $x_{-\nu,p}$ is the idele associated to $-\nu$ defined in Remark \ref{e13.10} and $\sigma_p^{k,k}$ is as in equation (\ref{e12.1}).
\label{T.intro.2}
\end{theorem} 

The proof uses three ingredients, namely: (a) the main theorem in \cite{chris2017}, i.e. the construction and interpolation of $L_p(\mathcal{F},-)$ for a Bianchi modular form $\mathcal{F}$; (b) the complex functional equation obtained in Theorem \ref{T.intro.1} and (c) the work done in \cite{loeffler2014p} that uniquely determines $L_p(\mathcal{F},-)$ by its values on the $p$-adic characters $\psi_{p-\mathrm{fin}}$ coming from a Hecke character $\psi$ as in Theorem \ref{T.intro.1} with conductor $\mathfrak{f}|p^\infty$ (see Section \ref{constructing} for the definition of $\psi_{p-\mathrm{fin}}$), when $\mathcal{F}$ has small slope at every $\pprime|p$.

\begin{remark}
In Theorem \ref{T.intro.1} the level of $\mathcal{F}$ must be coprime with the conductor of $\psi$. On the other hand, for the $p$-adic setting of Theorem \ref{T.intro.2}, $p$ needs to be in the level and also its proof use Hecke characters $\psi$ with conductor $\mathfrak{f}|p^\infty$ then the level and the conductor are not coprime. As a consequence we are forced to work first with Bianchi newforms of prime-to-$p$ level, and then successively stabilise at each prime $\pprime|p$, consequently missing Bianchi newforms at $p$.
\end{remark}

The construction of the $p$-adic $L$-function in \cite{chris2017} and then the functional equation in Theorem \ref{T.intro.2} depend of the small slope condition of the Bianchi modular form $\mathcal{F}$. It is natural to ask for the $p$-adic $L$-function when $\mathcal{F}$ does not have small slope, i.e. the critical slope case. In \cite{salazar2018} such function was constructed for certain base-change Bianchi modular forms. We briefly describe the construction. 

Let $f\in S_{k+2}(\Gamma_0(N))$ be a finite slope eigenform, with $p|N$, new or $p$-stabilised of a newform, regular, non CM by $K$, decent and such that the base-change to $K$, denoted by $f_{/K}$, is $\Sigma$-smooth (see Conditions \ref{conditions} for more details) and let $V_\mathbb{Q}$ be a neighbourhood of $f$ such that the weight map $w$ is étale except possibly at $f$. Then, after shrinking $V_\mathbb{Q}$, in \cite{salazar2018} was constructed the \textit{three-variable $p$-adic $L$-function} 
\begin{equation*}
\mathcal{L}_p: V_{\mathbb{Q}}\times \mathfrak{X}(\mathrm{Cl}_K(p^{\infty})) \rightarrow L 
\end{equation*}
for sufficiently large $L\subset\overline{\mathbb{Q}}_p$, such that for any classical point $y \in V_\mathbb{Q}(L)$ with small slope base-change $f_{y/K}$ we have $\mathcal{L}_p(y,-)=c_yL_p(f_{y/K},-)$, where $c_y\in L^\times$ is a $p$-adic period at $y$ and $L_p(f_{y/K},-)$ is the $p$-adic $L$-function of $f_{y/K}$. Now, suppose that $f_{/K}$ has critical slope and is $\Sigma$-smooth; then the missing $p$-adic $L$-function of $f_{/K}$ is defined to be the specialisation $L_p({f_{/K}},-):=\mathcal{L}_p(x_f,-)$, where $x_f\in V_\mathbb{Q}$ is the point corresponding to $f$.

Our second result is the functional equation of $L_p(f_{/K},-)$, in  particular with no non-critical assumption on $f_{/K}$.

We can transfer the functional equation in Theorem \ref{T.intro.2} to $\mathcal{L}_p$ by shrinking $V_\mathbb{Q}$ such that there exist a Zariski-dense set $S\subset V_{\mathbb{Q}}$ of classical points $y$ satisfying that $f_{y/K}$ is a small slope successively stabilisation at every $\pprime|p$ of a Bianchi newform in $S_{(k_y,k_y)}(\Gamma_0(\mathfrak{n}))$ with $\mathfrak{n}=(\nu)$ being the prime-to-$p$ part of the level of $f_{/K}$ and $k_y\equiv k \pmod{p-1}$. Then we obtain:

\begin{theorem}
For every $y\in V_\mathbb{Q}$
and all $\kappa\in\mathfrak{X}(\mathrm{Cl}_K(p^{\infty}))$ we have
\begin{equation*}
\mathcal{L}_p(y,\kappa)=-\epsilon(\mathfrak{n})w_{\mathrm{Tm}}(\mathrm{N}(\mathfrak{n}))^{k/2}  \langle \mathrm{N}(\mathfrak{n}) \rangle^{k_y/2}\kappa(x_{-\nu,p})^{-1}\mathcal{L}_p(y,\kappa^{-1}w_{\mathrm{Tm}}^{k}\langle \cdot \rangle^{k_y}).
\label{T.intro.3}
\end{equation*}
\end{theorem}
Where $\omega_{\mathrm{Tm}}=\prod_{\pprime|p}\omega_{\mathrm{Tm},\pprime}$ with $\omega_{\mathrm{Tm},\pprime}$ denoting the Teichmüller character at $\pprime$ and $\langle x \rangle:=\omega_{\mathrm{Tm}}(x)^{-1}x$ for $x\in(\tensorspace)^\times$.

Finally, since the $p$-adic $L$-function of a critical slope $\Sigma$-smooth base-change $f_{/K}$ is defined to be the specialisation $L_p({f_{/K}},-):=\mathcal{L}_p(x_f,-)$, then from specialising in Theorem \ref{T.intro.3} we obtain the following:
\begin{corollary}
Let $\mathcal{F}$ be a $\Sigma$-smooth base-change to $K$ of a modular form satisfying Conditions \ref{conditions}, let $\mathfrak{n}=(\nu)$ be the prime-to-$p$ part of the level of $\mathcal{F}$, then for all $\kappa\in\mathfrak{X}(\mathrm{Cl}_K(p^{\infty}))$ the distribution $L_p(\mathcal{F},-)$ satisfies the following functional equation
\begin{equation*}
L_p(\mathcal{F},\kappa)=-\epsilon(\mathfrak{n})\mathrm{N}(\mathfrak{n})^{k/2}\kappa(x_{-\nu,p})^{-1}L_p(\mathcal{F},\kappa^{-1}\sigma_p^{k,k}).
\end{equation*}
\end{corollary}

Notice that corollary above not only gives the functional equation for $\Sigma$-smooth critical slope base-change Bianchi modular forms, but also for $\Sigma$-smooth small slope base-change Bianchi newforms at $p$, this case is interesting, considering for example, that when $p$ is split in $K$, small slope is automatic.

\section*{Acknowledgments}

I would like to thank my PhD supervisor Daniel Barrera for suggesting this topic to me, as well as for the many conversations we’ve had on the subject. Thanks also to Chris Williams for helpful conversations about Bianchi modular forms. Finally, I like to thank the referee for their valuable comments and corrections. This work was funded by the National Agency for Research and Development (ANID, Chile)/Scholarship Program/BECA DOCTORADO NACIONAL/2018 - 21180506.

\section{Bianchi modular forms}

\subsection{Notation} \hfill\\

Throughout this paper, we fix $p$ a rational prime and take $K$ to be an imaginary quadratic field with class number 1 and ring of integers $\mathcal{O}_K$, let $\delta=\sqrt{-D}$ (where $-D$ is the discriminant of $K$) be a generator of the different ideal $\mathcal{D}$ of $K$, $\mathfrak{n}=(\nu)$ an ideal of $\mathcal{O}_K$. At each prime $\mathfrak{q}$ of $K$, denote by $K_{\mathfrak{q}}$ the completion of $K$ with respect to $\mathfrak{q}$, $\mathcal{O}_{\mathfrak{q}}$ the ring of integers of $K_{\mathfrak{q}}$ and fix a uniformiser $\pi_\mathfrak{q}$ at $\mathfrak{q}$. Denote the adele ring of $K$ by $\mathbb{A}_K=K_{\infty}\times\mathbb{A}_K^f$ where $K_{\infty}$ are the infinite adeles and $\mathbb{A}_K^f$ are the finite adeles. Furthermore, define $\widehat{\mathcal{O}_K}:=\mathcal{O}_K\otimes_{\mathbb{Z}}\widehat{\mathbb{Z}}$ to be the finite integral adeles. Let $n\geqslant0$ be an integer and denote by $V_{n}(R)$ the space of homogeneous polynomials in two variables of degree $n$ over a ring $R$. Note that $V_{n}(\mathbb{C})$ is an irreducible complex right representation of $\mathrm{SU}_2(\mathbb{C})$, denote it by $\rho_{n}$.

For a general Hecke character $\psi$ of $K$, for each prime $\mathfrak{q}$ of $K$ we denote by $\psi_{\mathfrak{q}}$ the restriction of $\psi$ to $K_{\mathfrak{q}}^\times$ and for an ideal $I\subset\mathcal{O}_K$, we define $\psi_I=\prod_{\mathfrak{q}|I}\psi_\mathfrak{q}$; we also write $\psi_\infty$ for the restriction of $\psi$ to the infinite ideles, and $\psi_f$ for the restriction to the finite ideles.

\subsection{Background}\hfill\\

Let $\Omega_0(\mathfrak{n})= \left\{ \smallmatrixx{a}{b}{c}{d} \in \mathrm{GL_2}(\widehat{\mathcal{O}_K}): c \in \mathfrak{n}\widehat{\mathcal{O}_K} \right\}$ and let $\varphi$ be a Hecke character whose conductor divides $\mathfrak{n}$ and with infinity type $(-k-2v_1,-k-2v_2)$ for $k\geq0$, $v_1$, $v_2$ integers. For $u_f=\smallmatrixx{a}{b}{c}{d}\in\Omega_0(\mathfrak{n})$ we set $\varphi_{\mathfrak{n}}(u_f)=\varphi_{\mathfrak{n}}(d)=\prod_{\mathfrak{q}|\mathfrak{n}}\varphi_{\mathfrak{q}}(d_{\mathfrak{q}})$.% with $\varphi_\mathfrak{q}(d_{\mathfrak{q}})$ trivial if $\mathfrak{q}\nmid\mathfrak{n}$.

\begin{definition}
We say a function $\Phi:\mathrm{GL_2}(\mathbb{A}_K) \rightarrow V_{2k+2}(\mathbb{C})$ is a \textit{cuspidal automorphic form over $K$} of weight $\lambda=[(k,k),(v_1,v_2)]$, level $\Omega_0(\mathfrak{n})$ and central action $\varphi$ if it satisfies:
\begin{enumerate}
\item[(i)] $\Phi$ is left-invariant under $\mathrm{GL_2}(K)$; 
\item[(ii)] $\Phi(zg)=\varphi(z)\Phi(g)$ for $z\in \mathbb{A}_K^\times \cong Z(\mathrm{GL_2}(\mathbb{A}_K))$, where $Z(G)$ denote the centre of the group $G$; 
\item[(iii)] $\Phi(gu)=\varphi_\mathfrak{n}(u_f)\Phi(g)\rho_{2k+2}(u_\infty)$ for $u=u_f\cdot u_\infty \in \Omega_0(\mathfrak{n})\times\mathrm{SU_2}(\mathbb{C})$;
\item[(iv)] $\Phi$ is an eigenfunction of the operator $\partial$, where $\partial /4$ denotes a component of the Casimir operator in the Lie algebra $\mathfrak{sl}_2(\mathbb{C})\otimes\mathbb{C}$, and where we consider $\Phi(g_{\infty}g_f)$ as a function of $g_\infty \in \mathrm{GL_2}(\mathbb{C}).$
\item[(v)] $\Phi$ satisfies the cuspidal condition that for all $g \in \mathrm{GL_2}(\mathbb{A}_K)$, we have 
\begin{equation*}
\int_{K\backslash\mathbb{A}_K}\Phi\left(\smallmatrixx{1}{u}{0}{1}g\right)du=0
\end{equation*} 
where $du$ is the Lebesgue measure on $\mathbb{A}_K$.
\end{enumerate}
The space of such functions will be denoted by $S_{\lambda}(\Omega_0(\mathfrak{n}),\varphi)$.
\label{T1.4}
\end{definition}

\begin{remark}
It is possible to define automorphic forms over $K$ of weight $\lambda=[(k_1,k_2),(v_1,v_2)]$, for distinct integers $k_1$ and $k_2$ but since we are dealing with cusp forms we assume $k_1=k_2=k$ (see \cite[\S2.5, Cor 2.2]{hida1994critical}).
\end{remark}

A cuspidal automorphic form $\Phi$ of weight $\lambda$ and level $\Omega_0(\mathfrak{n})$ descends to give a function $F:\mathrm{GL_2}(\mathbb{C})\rightarrow V_{2k+2}(\mathbb{C})$, via $F(g):=\Phi(g)$ for $g\in\mathrm{GL}_2(\mathbb{C})\subset\mathrm{GL_2}(\mathbb{A}_K)$. 

Let $\mathcal{H}_3:=\{(z,t): z\in\mathbb{C}, t\in\mathbb{R}_{>0}\}$ be the hyperbolic space and since $\mathrm {GL_2}(\mathbb{C})=Z(\mathrm{GL_2}(\mathbb{C}))\cdot \mathrm{B} \cdot \mathrm{SU_2}(\mathbb{C})$, where $\mathrm{B}=\left\{ \smallmatrixx{t}{z}{0}{1}: z\in\mathbb{C}, t\in\mathbb{R}_{>0} \right\}\cong \mathcal{H}_3$, we can descend further using ii) and iii) in Definition \ref{T1.4} to obtain a function
\begin{align}
\mathcal{F}&: \mathcal{H}_3\longrightarrow V_{2k+2}(\mathbb{C}), \nonumber\\
& (z,t) \longmapsto t^{v_1+v_2-1}F\smallmatrixx{t}{z}{0}{1}.
\label{e.3.1}
\end{align}

\begin{remark}
There are two ways to define the function $\mathcal{F}$ in the literature.
\begin{enumerate}
\item In accounts such as \cite{cremona1981}, \cite{cremona1994} and \cite{bygott1998modular}, all of which deal predominantly with weight $[(0,0),(0,0)]$ and trivial central action; such function is defined simply restricting to B.
\item In \cite{ghate1999critical} and \cite{chris2017} such $\mathcal{F}$ is defined as in (\ref{e.3.1}) for weight $[(k,k),(0,0)]$.
\end{enumerate}
For the two definitions above, the modularity condition satisfied by such $\mathcal{F}$ resulting is different, but there is a clear bijection between the sets of functions that arise.
\end{remark}

\begin{definition}
A function $\mathcal{F}:\mathcal{H}_3\rightarrow V_{2k+2}(\mathbb{C})$ is a \textit{cuspidal Bianchi modular form} if comes from $\Phi$, a cuspidal automorphic form over $K$ (in the sense of definition \ref{T1.4}) by the descent described above.
\end{definition}

Note that for $\gamma=\smallmatrixx{a}{b}{c}{d}\in\Gamma_0(\mathfrak{n}):=\mathrm{SL_2}(K)\cap\Omega_0(\mathfrak{n})\mathrm{GL}_2(\mathbb{C})$, a Bianchi modular form $\mathcal{F}$ satisfies the following automorphic condition:
\begin{equation}
\mathcal{F}(\gamma\cdot(z,t))=\varphi_\mathfrak{n}(d)^{-1}\mathcal{F}(z,t)\rho_{2k+2}(J(\gamma;(z,t))),
\label{e2.4}
\end{equation}
where $\gamma\cdot(z,t)$ denotes the action of $\rm{GL}_2(\mathbb{C})$ on $\mathcal{H}_3$ and $J(\gamma;(z,t)):=\smallmatrixx{cz+d}{\overline{c}t}{-ct}{\overline{cz+d}}$.

We denote by $S_{\lambda}(\Gamma_0(\mathfrak{n}),\varphi_\mathfrak{n}^{-1})$ the space of cuspidal Bianchi modular forms $\mathcal{F}$ that come from cuspidal automorphic forms $\Phi\in S_{\lambda}(\Omega_0(\mathfrak{n}),\varphi)$ and then they satisfy (\ref{e2.4}). Also denote by $S_{\lambda}(\Gamma_0(\mathfrak{n}))$ when $\varphi$ has trivial conductor.

\begin{definition}
Let $\gamma\in \mathrm{GL_2}(\mathbb{C})$ and let $\mathcal{F}\in S_{\lambda}(\Gamma_0(\mathfrak{n}),\varphi_\mathfrak{n}^{-1})$, then define a new function $\mathcal{F}|_\gamma$ by
\begin{equation}
(\mathcal{F}|_\gamma)(z,t):=det(\gamma)^{-k/2-v_1}\overline{det(\gamma)}^{-k/2-v_2}\mathcal{F}(\gamma \cdot(z,t))  \rho_{2k+2}^{-1}\left(J\left(\frac{\gamma}{\sqrt{det(\gamma)}} ;(z,t)\right)\right).
\label{e2.5}
\end{equation}
\end{definition}

\begin{remark}
Note that $\mathcal{F}\in S_{\lambda}(\Gamma_0(\mathfrak{n}),\varphi_\mathfrak{n}^{-1})$ satisfies:

1) $\mathcal{F}|_g(0,1)=F(g)$ for $g\in\mathrm{GL}_2(\mathbb{C})$, and for $g=\smallmatrixx{t}{z}{0}{1}\in\mathrm{B}$ we obtain (\ref{e.3.1}).

2) $\left(\mathcal{F}|_\gamma\right)(z,t)=\varphi_\mathfrak{n}(d)^{-1}\mathcal{F}(z,t)$ for $\gamma=\smallmatrixx{a}{b}{c}{d} \in \Gamma_0(\mathfrak{n})$.
\label{rem2.6}
\end{remark}

\subsection{Fourier-Whittaker expansions}\hfill\\

Let $\Phi: \mathrm{GL}_2(\mathbb{A}_K) \rightarrow V_{2k+2}(\mathbb{C})$ be a cuspidal automorphic form of weight $\lambda=[(k,k),(v_1,v_2)]$. Then $\Phi$ has the following Fourier expansion (see \cite[Thm 6.1]{hida1994critical})
\begin{equation}
\Phi\left[\smallmatrixx{t}{z}{0}{1}\right]=|t|_{\mathbb{A}_K} \sum_{\alpha \in K^\times} c(\alpha t \delta,\Phi)W(\alpha t_\infty) e_K(\alpha z) \;\;\text{for}\;\; \smallmatrixx{t}{z}{0}{1} \in\mathrm{GL}_2(\mathbb{A}_K),
\label{e.2.1}
\end{equation}
where:
\begin{itemize}
\item[i)] the Fourier coefficient $c(\cdot,\Phi)$ is a function on the fractional ideals of $K$, with $c(I,\Phi)=0$ for $I$ non-integral;
\item[ii)] $e_K$ is an additive character of $K\backslash \mathbb{A}_K$ defined by
\begin{equation*}
e_K=\left( \prod_{\pprime \;\mathrm{prime}} (e_p \circ \mathrm{Tr}_{K_\pprime/\mathbb{Q}_p})\right)\cdot (e_\infty \circ \mathrm{Tr}_{\mathbb{C}/\mathbb{R}}),  
\end{equation*}
for 
\begin{equation*}
e_p\left( \sum_j d_jp^j \right)=e^{-2\pi i \sum_{j<0} d_jp^j} \;\;\;  \mathrm{and} \;\;\; e_\infty(r)=e^{2\pi ir};    
\end{equation*}
\item[iii)] $W:\mathbb{C}^\times \rightarrow V_{2k+2}(\mathbb{C})$ is the Whittaker function
\begin{equation*}
W(s):= \sum_{n=0}^{2k+2} \matrixxx{2k+2}{n}\frac{1}{s^{v_1}\overline{s}^{v_2}} \left( \frac{s}{i|s|} \right)^{k+1-n} K_{n-(k+1)}(4 \pi |s|)X^{2k+2-n}Y^n,    
\end{equation*}
where $K_n(x)$ is a modified Bessel function.
\end{itemize}

If our cuspidal automorphic form, $\Phi$, corresponds to a cuspidal Bianchi modular form $\mathcal{F}$ on $\mathcal{H}_3$, then the Fourier expansion stated above descends to the following Fourier expansion of $\mathcal{F}$ (see \cite{ghate1999critical}):

\begin{equation*}
\mathcal{F}(z,t)\matrixxx{X}{Y}= \sum_{n=0}^{2k+2} \mathcal{F}_n(z,t) X^{2k+2-n}Y^{n},   
\end{equation*}

\begin{equation}
\mathcal{F}_n(z,t):=t\matrixxx{2k+2}{n} \sum_{\alpha\in K^{\times}} \left[ c(\alpha\delta)\frac{1}{\alpha^{v_1}\overline{\alpha}^{v_2}}\left(\frac{\alpha}{i|\alpha|}\right)^{k+1-n} K_{n-k-1}(4\pi|\alpha|t)e^{2\pi i (\alpha z + \overline{\alpha z})}\right].
\label{e1.6}
\end{equation}
Here to ease notation we have written $c(\alpha\delta)$ instead $c(\alpha\delta,\Phi)$.

\subsection{Hecke operators}\label{section2.4}\hfill\\

As with classical modular forms, we can extend the action of $\mathrm{GL}_2(\mathbb{C})$ on functions, given by (\ref{e2.5}), to the group ring of $\mathrm{GL}_2(\mathbb{C})$; Hecke operators will be defined by particular elements on this group ring.

Let $\mathfrak{q}\nmid\mathfrak{n}$ be a prime ideal of $\mathcal{O}_K$ generated by the fixed uniformiser $\pi_\mathfrak{q}$. Let $\mathcal{F}\in S_{\lambda}(\Gamma_0(\mathfrak{n}),\varphi_\mathfrak{n}^{-1})$ be a Bianchi modular form of weight $\lambda=[(k,k),(v_1,v_2)]$ we define the Hecke operator
\begin{equation*}
\mathcal{F}\mapsto \mathcal{F}|_{T_{\mathfrak{q}}}:=\pi_{\mathfrak{q}}^{k+2v_1}\overline{\pi_{\mathfrak{q}}}^{k+2v_2}\left[\sum_{b\in(\mathcal{O}_K/\mathfrak{q})^\times} \mathcal{F}|_{\bonitaa{1}{b}{0}{\pi_{\mathfrak{q}}}} + \varphi_{\mathfrak{n}}(\pi_{\mathfrak{q}})^{-1}\mathcal{F}|_{\bonitaa{\pi_{\mathfrak{q}}}{0}{0}{1}}\right].
\end{equation*}
When $\mathfrak{q}|\mathfrak{n}$ we denote $T_\mathfrak{q}$ by $U_\mathfrak{q}$ and 
\begin{equation*}
\mathcal{F}|_{U_{\mathfrak{q}}}:=\pi_{\mathfrak{q}}^{k+2v_1}\overline{\pi_{\mathfrak{q}}}^{k+2v_2}\sum_{b\in(\mathcal{O}_K/\mathfrak{q})^\times} \mathcal{F}|_{\bonitaa{1}{b}{0}{\pi_{\mathfrak{q}}}}.
\end{equation*}

We can similarly define Hecke operators for each ideal $I$ of $K$. Indeed, let $I=\prod_{\mathfrak{q}}\mathfrak{q}^r$ where $\mathfrak{q}^r$ exactly divides $I$, then the Hecke operator $T_I$ is totally determined by the Hecke operators $T_\mathfrak{q}$ for $\mathfrak{q}|I$.

\begin{remark}
The Hecke operators on Bianchi modular forms defined above are the descent of the Hecke operators defined in \cite[Chap.VI]{weil1971dirichlet}) on automorphic forms over $K$ by the action of double cosets.
\end{remark}

In the same way as in the rational case (elliptic modular forms), for Bianchi modular forms of weight $\lambda=[(k,k),(v_1,v_2)]$ and level $\Gamma_0(\mathfrak{n})$ with $\mathfrak{n}=(\nu)$ there is a Fricke involution $W_{\mathfrak{n}}$ defined by
\begin{equation*}
\mathcal{F}|_{W_{\mathfrak{n}}}:=\nu^{k/2+v_1}\overline{\nu}^{k/2+v_2}\mathcal{F}|_{\bonitaa
{0}{-1}{\nu}{0}}.
\end{equation*}

\begin{lemma}
Let $\mathcal{F}\in S_{\lambda}(\Gamma_0(\mathfrak{n}),\varphi_\mathfrak{n}^{-1})$ with $\lambda=[(k,k),(v_1,v_2)]$,  
then for $0\leqslant n \leqslant 2k+2$ we have
\begin{equation*}
(\mathcal{F}|_{W_{\mathfrak{n}}})_{2k+2-n}(0,t)=t^{-2k-2}(-1)^n\nu^{n-k-1}|\nu|^{-n}\mathcal{F}_{n}\left( 0,\frac{1}{|\nu|t} \right).
\end{equation*}
\label{T1.19}
\end{lemma}{}

\begin{proof}
Note that if $\gamma=\smallmatrixx{0}{-1}{\nu}{0}$ then $\gamma\cdot(0,t)=(0,1/(|\nu|t))$ and

\begin{align*}
\rho_{2k+2}^{-1}\left(J\left(\frac{\gamma}{\sqrt{det(\gamma)}} ;(0,t)\right)\right)\matrixxx{X}{Y}^{2k+2}&=
\left(J\left(\frac{\gamma}{\sqrt{det(\gamma)}} ;(0,t)\right)^{-1}\matrixxx{X}{Y}\right)^{2k+2}\\
&=\left(\matrixx{0}{-\nu^{-1/2}t^{-1}}{\overline{\nu}^{-1/2}t^{-1}}{0}\matrixxx{X}{Y}\right)^{2k+2}\\
&=\matrixxx{-\nu^{-1/2}t^{-1}Y}{\overline{\nu}^{-1/2}t^{-1}X}^{2k+2}
\end{align*}
where $\matrixxx{X}{Y}^{2k+2}=(X^{2k+2}, X^{2k+1}Y,...,X^{2k+2-n}Y^n,..., XY^{2k+1}, Y^{2k+2})^t$.\\

Then,
\begin{align*}
(\mathcal{F}|_{W_{\mathfrak{n}}})(0,t)\matrixxx{X}{Y}^{2k+2}&=\nu^{k/2+v_1}\overline{\nu}^{k/2+v_2}\mathcal{F}|_{\bonitaa{0}{-1}{\nu}{0}}(0,t)\matrixxx{X}{Y}^{2k+2}\\
&=\mathcal{F}(0,1/(|\nu|t))\cdot \matrixxx{-\nu^{-1/2}t^{-1}Y}{\overline{\nu}^{-1/2}t^{-1}X}^{2k+2}.
\end{align*}
\end{proof}

\subsection{Twisted series of Bianchi modular forms}\label{sectiontwist}\hfill\\

For the rest of the paper we fix the weight $\lambda=[(k,k),(0,0)]$, i.e. $v_1=v_2=0$ as in \cite{chris2017}, then we say $\Phi$ has weight $(k,k)$ and denote $S_{\lambda}(\Omega_0(\mathfrak{n}),\varphi)$ by $S_{(k,k)}(\Omega_0(\mathfrak{n}),\varphi)$ and $S_{\lambda}(\Gamma_0(\mathfrak{n}),\varphi_\mathfrak{n}^{-1})$ by $S_{(k,k)}(\Gamma_0(\mathfrak{n}),\varphi_\mathfrak{n}^{-1})$. The importance of the general definition with $(v_1,v_2)$ in previous sections will become clear in this section when we twist Bianchi modular forms by Hecke characters.

\begin{definition}
Let $\Phi\in S_{(k,k)}(\Omega_0(\mathfrak{n}),\varphi)$ and $\psi$ be a Hecke character of conductor $\mathfrak{f}$. Define the twisting operator $R(\psi)$ by
\begin{equation*}
\Phi|R(\psi)(g):=\psi(det(g))\sum_{[a]\in(\mathfrak{f}^{-1}/\mathcal{O}_K)^\times}\psi_{\mathfrak{f}}(a)\Phi(g\smallmatrixx{1}{a}{0}{1}),\;\;\text{for}\;\;g\in\mathrm{GL}_2(\mathbb{A}_K).    
\end{equation*}
\label{def2.9}
\end{definition}

\begin{proposition}
Let $\Phi\in S_{(k,k)}(\Omega_0(\mathfrak{n}),\varphi)$ where $\varphi$ has infinity type $(-k,-k)$ and conductor dividing $\mathfrak{n}$ and let $\psi$ be a Hecke character of infinity type $(q,r)$ and conductor $\mathfrak{f}$. Then $\Phi|R(\psi)\in S_{\iota}(\Omega_0(\mathfrak{m}),\varphi\psi^2)$ where $\iota=[(k,k),(-q,-r)]$ and $\mathfrak{m}=\mathfrak{n}\cap\mathfrak{f}^2$.
\label{Prop2.10}
\end{proposition}

\begin{proof}
See \cite[\S6, (6.7)]{hida1994critical}.
\end{proof}

\begin{remark}
Note that for $\mathfrak{f}=(f)$, if $\mathcal{F}$ is the descent of a cuspidal automorphic form $\Phi\in S_{(k,k)}(\Omega_0(\mathfrak{n}),\varphi)$ and we denote by $\mathcal{F}_{\psi}$ the descent of $\Phi|R(\psi)$ then by \cite[(6.9)]{hida1994critical}, we have
\begin{equation*}
\mathcal{F}_{\psi}=\sum_{b\in(\mathcal{O}_K/\mathfrak{f})^\times}  \psi_\mathfrak{f}(b/f)\mathcal{F}|_{\bonitaa{1}{b/f}{0}{1}}=\psi_\infty(f)\sum_{b\in(\mathcal{O}_K/\mathfrak{f})^\times}  \psi_\mathfrak{f}(b)\mathcal{F}|_{\bonitaa{1}{b/f}{0}{1}},
\label{R6.5}
\end{equation*}
which is analogous to the twist of a modular form by a Dirichlet character in \cite[Thm 7.4 (7.30)]{iwaniec1997topics} up to a factor of a Gauss sum (and $\psi_\infty$) which then appears in the Fourier expansion of $\mathcal{F}$ (see \cite[(6.8)]{hida1994critical}).
\label{remark6.5}
\end{remark}

\begin{lemma}
If $\mathcal{F}\in S_{(k,k)}(\Gamma_0(\mathfrak{n}),\varphi_\mathfrak{n}^{-1})$ then $\mathcal{F}|_{W_\mathfrak{n}}\in S_{(k,k)}(\Gamma_0(\mathfrak{n}),\varphi_\mathfrak{n})$.
\label{T2.3}
\end{lemma}

\begin{proof}
Note that $\smallmatrixx{0}{-1}{\nu}{0}$ normalizes the group $\Gamma_0(\mathfrak{n})$, explicitly $\smallmatrixx{0}{-1}{\nu}{0}\gamma=\gamma'\smallmatrixx{0}{-1}{\nu}{0}$ where $\gamma'=\smallmatrixx{d}{-c/\nu}{-b\nu}{a}$ if $\gamma=\smallmatrixx{a}{b}{c}{d}$. Hence, for $\mathcal{F}\in S_{(k,k)}(\Gamma_0(\mathfrak{n}),\varphi_\mathfrak{n}^{-1})$ and $\gamma\in\Gamma_0(\mathfrak{n})$ we have
\begin{equation*}
(\mathcal{F}|_{W_\mathfrak{n}})|_\gamma =|\nu|^{k}\mathcal{F}|_{\bonitaa{0}{-1}{\nu}{0}\gamma} =|\nu|^{k}\mathcal{F}|_{\gamma'{\bonitaa{0}{-1}{\nu}{0}}} =\varphi_\mathfrak{n}(\gamma')^{-1}\mathcal{F}|_{W_\mathfrak{n}} =\varphi_\mathfrak{n}(\gamma)\mathcal{F}|_{W_\mathfrak{n}},  
\end{equation*}
where in last equality we use that $ad\equiv1\mod{\mathfrak{n}}$.
\end{proof}

\begin{proposition}
Let $\mathcal{F}\in S_{(k,k)}(\Gamma_0(\mathfrak{n}),\varphi_\mathfrak{n}^{-1})$ be a Bianchi modular form and let $\psi$ be a Hecke character of conductor $\mathfrak{f}$ with $(\mathfrak{n},\mathfrak{f})=1$. Then
\begin{equation*}
\mathcal{F}_{\psi}|_{W_\mathfrak{m}}=\varphi_\mathfrak{n}(f)^{-1}\psi_\mathfrak{f}(-\nu)^{-1}\psi_\infty(f)^2\left(\mathcal{F}|_{W_\mathfrak{n}}\right)_{\psi^{-1}},
\end{equation*}
for $(m)=\mathfrak{m}=\mathfrak{n}\mathfrak{f}^2=(\nu)(f)^2$, $m=\nu f^2$.
\label{T6.7}
\end{proposition}

\begin{proof}
Since for any $v$ we have the identity $\bonitaa{1}{\frac{b}{f}}{0}{1}
\bonitaa{0}{-1}{m}{0} = \scalebox{0.8}{%
$f$} \cdot \bonitaa{0}{-1}{\nu}{0} \bonitaa{f}{-v}{-b\nu}{\frac{1+bv\nu}{f}} \bonitaa{1}{\frac{v}{f}}{0}{1}$ and choosing $v$ such that $bv\nu\equiv-1$ (mod$f$) to bring $\smallmatrixx{f}{-v}{-b\nu}{\frac{1+bv\nu}{f}}$ into $\Gamma_0(\mathfrak{n})$, then 
\begin{equation}
\begin{array}{ll}
\mathcal{F}|_{\bonitaa{1}{\frac{b}{f}}{0}{1}W_\mathfrak{m}}
&=|m|^{k}\mathcal{F}|_{\bonitaa{1}{\frac{b}{f}}{0}{1}\bonitaa{0}{-1}{m}{0}}
=|m|^{k}\mathcal{F}|_{\scalebox{0.8}{%
$f$}\cdot\bonitaa{0}{-1}{\nu}{0}\bonitaa{f}{-v}{-b\nu}{\frac{1+bv\nu}{f}}\bonitaa{1}{\frac{v}{f}}{0}{1}} \\
&=|m|^{k}|f^2|^{-k}\mathcal{F}|_{\bonitaa{0}{-1}{\nu}{0} \bonitaa{f}{-v}{-b\nu}{\frac{1+bv\nu}{f}} \bonitaa{1}{\frac{v}{f}}{0}{1}}\\
&=|m|^{k}|f^2|^{-k}|\nu|^{-k}\left(\mathcal{F}|_{W_\mathfrak{n}}\right)|_{\bonitaa{f}{-v}{-b\nu}{\frac{1+bv\nu}{f}}\bonitaa{1}{\frac{v}{f}}{0}{1}}\\
&=\left(\mathcal{F}|_{W_\mathfrak{n}}\right)|_{\bonitaa{f}{-v}{-b\nu}{\frac{1+bv\nu}{f}} \bonitaa{1}{\frac{v}{f}}{0}{1}}=\varphi_\mathfrak{n}(f)^{-1}\left(\mathcal{F}|_{W_\mathfrak{n}}\right)|_{\bonitaa{1}{\frac{v}{f}}{0}{1}}.
\label{e2.1}
\end{array}
\end{equation}

Where in the last equalities we use that $m=\nu f^{2}$ and $\mathcal{F}|_{W_\mathfrak{n}}\in S_{(k,k)}(\Gamma_0(\mathfrak{n}),\varphi_\mathfrak{n})$ by Lemma \ref{T2.3}. Since we have that $bv\nu\equiv-1$ (mod$f$), then $\psi_\mathfrak{f}(b)=\psi_\mathfrak{f}(-\nu)^{-1}\psi_\mathfrak{f}(v)^{-1}$. Now multiplying (\ref{e2.1}) by the latter and summing over the reduced residue class of $(\mathcal{O}_K/\mathfrak{f})^\times$ we obtain
\begin{equation*}
\sum_{b\in(\mathcal{O}_K/\mathfrak{f})^\times}  \psi_\mathfrak{f}(b)\mathcal{F}|_{\bonitaa{1}{b/f}{0}{1}W_\mathfrak{m}}=\varphi_\mathfrak{n}(f)^{-1}\psi_\mathfrak{f}(-\nu)^{-1}\sum_{v\in(\mathcal{O}_K/\mathfrak{f})^\times}\psi_\mathfrak{f}(v)^{-1}\left(\mathcal{F}|_{W_\mathfrak{n}}\right)|_{\bonitaa{1}{\frac{v}{f}}{0}{1}}.
\end{equation*}
Multiplying by $\psi_\infty(f)$ in both sides and using Remark \ref{R6.5} we have
\begin{equation*}
\mathcal{F}_{\psi}|_{W_\mathfrak{m}}=\varphi_\mathfrak{n}(f)^{-1}\psi_\mathfrak{f}(-\nu)^{-1}\psi_\infty(f)\sum_{v\in(\mathcal{O}_K/\mathfrak{f})^\times}\psi_\mathfrak{f}(v)^{-1}\left(\mathcal{F}|_{W_\mathfrak{n}}\right)|_{\bonitaa{1}{\frac{v}{f}}{0}{1}}.
\end{equation*}
Finally, since $\psi_\mathfrak{f}(v)^{-1}=\psi_\infty(f)(\psi_\infty(f)^{-1}\psi_\mathfrak{f}(v)^{-1})=\psi_\infty(f)\psi_\mathfrak{f}(v/f)^{-1}$ we obtain the result.
\end{proof}

The proposition above is a generalization to the Bianchi setting of \cite[Thm 7.5]{iwaniec1997topics} up to a Gauss sum and $\psi_\infty$ as explained in Remark \ref{remark6.5}.

\section{$L$-function}\label{L-function}

\subsection{Definition of the $L$-function}\hfill\\

Let $\psi$ be a Hecke character with conductor $\mathfrak{f}$, for each ideal $\mathfrak{m}=\prod_{\mathfrak{q}|\mathfrak{m}}\mathfrak{q}^{n_{\mathfrak{q}}}$ coprime to $\mathfrak{f}$, we define $\psi(\mathfrak{m})=\prod_{\mathfrak{q}|\mathfrak{m}}\psi_{\mathfrak{q}}(\pi_{\mathfrak{q}})^{n_{\mathfrak{q}}}$ with $\mathfrak{q}=(\pi_{\mathfrak{q}})$ a prime ideal; and $\psi(\mathfrak{m})=0$ if $\mathfrak{m}$ is not coprime to $\mathfrak{f}$. In an abuse of notation, we write $\psi$ for both the idelic Hecke character and the function it determines on ideals; will always be clear from the context which formulation we mean.

Let $\Phi$ be an automorphic form, define the twist of the $L$-function of $\Phi$ by $\psi$ by
\begin{equation*}
L(\Phi,\psi,s)=\sum_{0\neq\mathfrak{m}\subset\mathcal{O}_K} c(\mathfrak{m},\Phi)\psi(\mathfrak{m})N(\mathfrak{m})^{-s},
\end{equation*}
where $c(\cdot,\Phi)$ are the Fourier coefficients of $\Phi$.

\begin{remark}
In \cite{weil1971dirichlet}, it is proved that the twisted $L$-function converges absolutely in some suitable right half-plane. The $L$-function can be written in terms of an integral formula, then via meromorphic continuation, this integral gives the definition of the $L$-function on all of $\mathbb{C}$. In fact, a little more work shows that this function is an \textit{analytic} continuation and the $L$-function is holomorphic on the whole complex plane.
\label{T1.13}
\end{remark}

Let $\mathcal{F}$ be a Bianchi modular form corresponding to the automorphic form $\Phi$, then
\begin{equation*}
L(\Phi,\psi,s)=L(\mathcal{F},\psi,s)=w^{-1}\sum_{\alpha\in K^{\times}} c(\alpha \delta)\psi((\alpha \delta))N((\alpha \delta))^{-s},
\end{equation*}
where $w=|\mathcal{O}_K^\times|$ and $c(\cdot)$ are the Fourier coefficients of $\mathcal{F}$ as in (\ref{e1.6}).

It is convenient to think the twisted $L$-function as a function on Hecke characters instead as a complex function of one variable. Then we put
\begin{equation*}
L(\mathcal{F},\psi)=L(\mathcal{F},\psi,1).   
\end{equation*}

We \textit{complete} the $L$-function by adding the appropriate factors at infinity. If the infinity type of $\psi$ is $(q,r)$ then we define
\begin{equation*}
\Lambda(\mathcal{F},\psi):=\frac{\Gamma(q+1)\Gamma(r+1)}{(2\pi i)^{q+1}(2\pi i)^{r+1}}L(\mathcal{F},\psi), 
\end{equation*}
where $\Gamma$ is the usual Gamma function. This is the $L$-function renormalised by Deligne's $\Gamma$-factors at infinity.

In \cite[Thm 2.11]{chris2017} is proved:

\begin{theorem}
Let $\mathcal{F}\in S_{(k,k)}(\Gamma_0(\mathfrak{n}))$ then for a Hecke character $\psi$ of conductor $\mathfrak{f}=(f)$ and infinity type $0 \leqslant (q,r) \leqslant (k,k)$ we have 
\begin{equation*}
\Lambda(\mathcal{F},\psi)= \frac{(-1)^{k+q+r}2}{Dw\tau(\psi^{-1})} \sum_{b\in(\mathcal{O}_K/f)^\times}\psi_{\mathfrak{f}}(b/f)c_{q,r}(b/f),
\end{equation*}
where $\tau(\psi^{-1})$ is the Gauss sum from \cite[\S 1.2.3.]{chris2017} and for $a\in K$
\begin{equation*}
c_{q,r}(a):=2\matrixxx{2k+2}{k+q-r+1}^{-1} (-1)^{k+r+1}\int_{0}^{\infty}t^{q+r}\mathcal{F}_{k+q-r+1}(a,t)dt.    
\label{T8.2}
\end{equation*}
\end{theorem}

\begin{remark}
Note that in Theorem \ref{T8.2}, the coefficient $c_{q,r}(a)$ differs from \cite[Prop. 2.9]{chris2017} where there is a slight error with the sub-index of $\mathcal{F}$. 
\end{remark}

In \cite[Thm. 8.1]{hida1994critical} is proved that the ``critical'' values of the $L$-function can be controlled, i.e. there exists a period $\Omega_\mathcal{F}\in\mathbb{C}^\times$ and a number field $E$ such that
\begin{equation}
\frac{\Lambda(\mathcal{F},\psi)}{\Omega_\mathcal{F}}\in E(\psi),
\label{e7.1}
\end{equation}
where $E(\psi)$ is the number field generated by the values of $\psi$.

\subsection{Functional equation of the complex $L$-function}\hfill\\

For the rest of the paper we work on the space $S_{(k,k)}(\Gamma_0(\mathfrak{n}))$, in particular we henceforth assume the central action $\varphi$ has trivial conductor. We also use the fact that if $\mathcal{F}\in S_{(k,k)}(\Gamma_0(\mathfrak{n}))$ is a Bianchi \textit{newform}, i.e., it is an eigenform for all the Hecke operators and is not induced from a Bianchi modular form with level properly dividing $\mathfrak{n}$; then $\mathcal{F}$ is an eigenvector for the Fricke involution $W_\mathfrak{n}$ with $\mathcal{F}|_{W_\mathfrak{n}}=\epsilon(\mathfrak{n})\mathcal{F}$ for $\epsilon(\mathfrak{n})=\pm1$ (see  \cite[\S2]{cremona1994}).

\begin{theorem}
Let $\mathcal{F}\in S_{(k,k)}(\Gamma_0(\mathfrak{n}))$ be a newform with $\mathfrak{n}=(\nu)$, then for a Hecke character $\psi$ of $K$ of conductor $\mathfrak{f}=(f)$ with $(f,\nu)=1$ and infinity type $0 \leqslant (q,r) \leqslant (k,k)$, we have 
\begin{equation*}
\Lambda(\mathcal{F},\psi)=\frac{-\epsilon(\mathfrak{n})|\nu|^k\tau(\psi|\cdot|_{\mathbb{A}_K}^{-k})}{\psi_\mathfrak{f}(-\nu)\psi_\infty(-\nu)\tau(\psi^{-1})}\Lambda(\mathcal{F},\psi^{-1}|\cdot|_{\mathbb{A}_K}^{k}).
\label{T16.1}
\end{equation*}
\end{theorem}

\begin{proof}
By Theorem \ref{T8.2} we know that for a Hecke character $\psi$ of $K$ of conductor $\mathfrak{f}=(f)$ and infinity type $0 \leqslant (q,r) \leqslant (k,k)$, we have 
\begin{align*}
\Lambda(\mathcal{F},\psi)&= \frac{(-1)^{k+q+r}2}{Dw\tau(\psi^{-1})} \sum_{b\in(\mathcal{O}_K/f)^\times}\psi_{\mathfrak{f}}(b/f)c_{q,r}(b/f)\\
&=\frac{(-1)^{k+q+r}2}{Dw\tau(\psi^{-1})} \sum_{b\in(\mathcal{O}_K/f)^\times}\psi_{\mathfrak{f}}(b/f)\left[\frac{2(-1)^{k+r+1}}{\matrixxx{2k+2}{k+q-r+1}} \int_{0}^{\infty}t^{q+r}\mathcal{F}_{k+q-r+1}(b/f,t)dt\right]\\
&=\frac{(-1)^{q+1}4}{Dw\tau(\psi^{-1})\matrixxx{2k+2}{k+q-r+1}}\int_{0}^{\infty}t^{q+r}\left[\sum_{b\in(\mathcal{O}_K/f)^\times}\psi_{\mathfrak{f}}(b/f) \mathcal{F}_{k+q-r+1}(b/f,t)\right]dt \\
&=\frac{(-1)^{q+1}4}{Dw\tau(\psi^{-1})\matrixxx{2k+2}{k+q-r+1}} \int_{0}^{\infty}t^{q+r}\mathcal{F}_{\psi,k+q-r+1}(0,t)dt,
\end{align*}
where last equality comes from Remark \ref{R6.5}.\\
Changing variable $t\rightarrow 1/(|m|t)$ we have 
\begin{equation*}
\Lambda(\mathcal{F},\psi)=\frac{(-1)^{q+1}4|m|^{-q-r-1}}{Dw\tau(\psi^{-1})\matrixxx{2k+2}{k+q-r+1}} \int_{0}^{\infty}t^{-q-r-2}\mathcal{F}_{\psi,k+q-r+1}(0,1/(|m|t))dt.
\end{equation*}
Recall that $\mathcal{F}_\psi$ has weight $[(k,k),(-q,-r)]$ then by Lemma \ref{T1.19} we know that
\begin{equation*}
(\mathcal{F}_{\psi}|_{W_\mathfrak{m}})_{k-q+r+1}(0,t)=t^{-2k-2}(-1)^{k+q-r+1}m^{q-r}|m|^{-(k+q-r+1)}\mathcal{F}_{\psi,k+q-r+1}(0,1/(|m|t))   
\end{equation*}
and replacing $\mathcal{F}_{\psi,k+q-r+1}(0,1/(|m|t))$ above we have
\begin{align*}
\Lambda(\mathcal{F},\psi)&=\frac{(-1)^{q+1}4|m|^{-q-r-1}}{Dw\tau(\psi^{-1})\matrixxx{2k+2}{k+q-r+1}}\\
&\times \int_{0}^{\infty}t^{-q-r-2}\left[t^{2k+2}(-1)^{k+q-r+1}m^{-q+r}|m|^{k+q-r+1}(\mathcal{F}_{\psi}|_{W_\mathfrak{m}})_{k-q+r+1}(0,t)\right]dt. \\
&=\frac{(-1)^{k+r}4m^{-q+r}|m|^{k-2r}}{Dw\tau(\psi^{-1})\matrixxx{2k+2}{k+q-r+1}} \int_{0}^{\infty}t^{2k-q-r}(\mathcal{F}_{\psi}|_{W_\mathfrak{m}})_{k-q+r+1}(0,t)dt. 
\end{align*}
By Proposition \ref{T6.7} since $(\nu,f)=1$ we have $\mathcal{F}_{\psi}|_{W_\mathfrak{m}}=\psi_\mathfrak{f}(-\nu)^{-1}\psi_\infty(f)^2(\mathcal{F}|_{W_\mathfrak{n}})_{\psi^{-1}}$, also $\mathcal{F}$ is a newform, then $\mathcal{F}_{\psi}|_{W_\mathfrak{m}}=\epsilon(\mathfrak{n})\psi_\mathfrak{f}(-\nu)^{-1}\psi_\infty(f)^2\mathcal{F}_{\psi^{-1}}$ and we have
\begin{align*}
&\Lambda(\mathcal{F},\psi)=\frac{(-1)^{k+r}4m^{-q+r}|m|^{k-2r}\epsilon(\mathfrak{n})\psi_{\infty}(f)^2}{Dw\tau(\psi^{-1})\matrixxx{2k+2}{k+q-r+1}\psi_\mathfrak{f}(-\nu)}  \int_{0}^{\infty}t^{2k-q-r}\mathcal{F}_{\psi^{-1},k-q+r+1}(0,t)dt\\  
&=\frac{(-1)^{k+r}4m^{-q+r}|m|^{k-2r}\epsilon(\mathfrak{n})\psi_{\infty}(f)^2}{Dw\tau(\psi^{-1})\matrixxx{2k+2}{k+q-r+1}\psi_\mathfrak{f}(-\nu)}  \int_{0}^{\infty}t^{2k-q-r}\left[\sum_{b\in(\mathcal{O}_K/f)^\times}\psi_{\mathfrak{f}}^{-1}(b/f) \mathcal{F}_{k-q+r+1}(b/f,t)\right]dt\\
&=\frac{(-1)^{k+r}4m^{-q+r}|m|^{k-2r}\epsilon(\mathfrak{n})\psi_{\infty}(f)^2}{Dw\tau(\psi^{-1})\matrixxx{2k+2}{k+q-r+1}\psi_\mathfrak{f}(-\nu)}\sum_{b\in(\mathcal{O}_K/f)^\times}\psi_{\mathfrak{f}}^{-1}(b/f)\int_{0}^{\infty}t^{2k-q-r} \mathcal{F}_{k-q+r+1}(b/f,t)dt.
\end{align*}

Now note that the integral above is exactly the integral appearing on the coefficient $c_{k-q,k-r}(b/f)$, more explicitly
\begin{equation*}
c_{k-q,k-r}(b/f)=\frac{2(-1)^{r+1}}{\matrixxx{2k+2}{k-q+r+1}} \int_{0}^{\infty}t^{2k-q-r}\mathcal{F}_{k-q+r+1}(b/f,t)dt
\end{equation*}
and since ${\matrixxx{2k+2}{k-q+r+1}}={\matrixxx{2k+2}{k+q-r+1}}$, we have
\begin{equation*}
\Lambda(\mathcal{F},\psi)=\frac{(-1)^{k+1}2m^{-q+r}|m|^{k-2r}\epsilon(\mathfrak{n})\psi_{\infty}(f)^2}{Dw\tau(\psi^{-1})\psi_\mathfrak{f}(-\nu)}\sum_{b\in(\mathcal{O}_K/f)^\times}\psi_{\mathfrak{f}}^{-1}(b/f)c_{k-q,k-r}(b/f).
\end{equation*}
The Hecke character $\psi^{-1}|\cdot|_{\mathbb{A}_K}^{k}$ has conductor $\mathfrak{f}$ and infinity type $(k-q,k-r)$ and we know by Theorem \ref{T8.2} that 
\begin{align*}
\Lambda(\mathcal{F},\psi^{-1}|\cdot|_{\mathbb{A}_K}^{k})&= \frac{(-1)^{k+q+r}2}{Dw\tau((\psi^{-1}|\cdot|_{\mathbb{A}_K}^{k})^{-1})} \sum_{b\in(\mathcal{O}_K/f)^\times}(\psi^{-1}|\cdot|_{\mathbb{A}_K}^{k})_{\mathfrak{f}}(b/f)c_{k-q,k-r}(b/f)\\
&= \frac{(-1)^{k+q+r}2|f|^{2k}}{Dw\tau(\psi|\cdot|_{\mathbb{A}_K}^{-k})} \sum_{b\in(\mathcal{O}_K/f)^\times}\psi^{-1}_{\mathfrak{f}}(b/f)c_{k-q,k-r}(b/f).
\end{align*}
Finally we obtain 
\begin{align*}
\Lambda(\mathcal{F},\psi)&=\frac{(-1)^{k+1}2m^{-q+r}|m|^{k-2r}\epsilon(\mathfrak{n})\psi_{\infty}(f)^2}{Dw\tau(\psi^{-1})\psi_\mathfrak{f}(-\nu)}\left[\frac{(-1)^{k+q+r}2|f|^{2k}}{Dw\tau(\psi|\cdot|_{\mathbb{A}_K}^{-k})}\right]^{-1}\Lambda(\mathcal{F},\psi^{-1}|\cdot|_{\mathbb{A}_K}^{k})  \\
&=\frac{(-1)^{q+r+1}m^{-q+r}|m|^{k-2r}\epsilon(\mathfrak{n})\psi_{\infty}(f)^{2}\tau(\psi|\cdot|_{\mathbb{A}_K}^{-k})}{\tau(\psi^{-1})\psi_\mathfrak{f}(-\nu)|f|^{2k}}\Lambda(\mathcal{F},\psi^{-1}|\cdot|_{\mathbb{A}_K}^{k})\\
&=\frac{-\epsilon(\mathfrak{n})|\nu|^k\tau(\psi|\cdot|_{\mathbb{A}_K}^{-k})}{\psi_\mathfrak{f}(-\nu)\psi_\infty(-\nu)\tau(\psi^{-1})}\Lambda(\mathcal{F},\psi^{-1}|\cdot|_{\mathbb{A}_K}^{k}),  
\end{align*}
where we used that 
\begin{align*}
\frac{m^{-q+r}|m|^{k-2r}\psi_{\infty}(f)^2}{|f|^{2k}}&=\frac{|m|^k\psi_\infty^{-1}(m)\psi_{\infty}(f^2)}{|f|^{2k}}=\frac{|\nu f^2|^k\psi_\infty^{-1}(\nu f^2)\psi_{\infty}(f^2)}{|f|^{2k}}\\
&=|\nu|^k\psi_\infty^{-1}(\nu)
\end{align*}
and $(-1)^{q+r}\psi_\infty^{-1}(\nu)=\psi_\infty^{-1}(-\nu)$.
\end{proof}

\begin{remark}
The theorem above generalizes the functional equation obtained in \cite[Prop 2.1]{cremona1994} for the untwisted $L$-function of a cuspidal Bianchi modular form of weight $(0,0)$. Also note that Theorem \ref{T16.1} is not a new result, but rather a reformulation of a classical result in \cite{JL}.
\end{remark}

\subsection{$p$-stabilisations of Bianchi modular forms}\hfill\\

Let $\mathcal{F}$ be a Bianchi eigenform of level $\mathfrak{n}$, the construction of the $p$-adic $L$-function of $\mathcal{F}$ requires $\mathfrak{n}$ to be divisible by each prime $\pprime$ above $p$, if this is not the case, we can define a \textit{$\pprime$-stabilisation} of $\mathcal{F}$ which will be a Bianchi eigenform with level at $\pprime$. In this section we define such $\pprime$-stabilisation using the descent to $\mathcal{H}_3$ of the general definition of a $p$-stabilisation in the adelic setting.

Let $\Phi$ be an automorphic eigenform over $K$ of weight $(k,k)$, level $\Omega_0(\mathfrak{n})$ with $\mathfrak{p}\nmid\mathfrak{n}$ and central action $\varphi$ of infinity type $(-k,-k)$ and trivial conductor. 
The Hecke operator $T_\mathfrak{q}$ for $\mathfrak{q}\nmid\mathfrak{n}$ is defined in \cite[Chap.VI, (8)]{weil1971dirichlet} by
\begin{equation*}
\Phi|_{T_\mathfrak{q}}(g):=\sum_{u\;\textrm{mod}\;\mathfrak{q}}\Phi\left(g\bonitaa{\pi_{\mathfrak{q}}}{u}{0}{1}\right)+\Phi\left(g\bonitaa{1}{0}{0}{\pi_{\mathfrak{q}}}\right).   
\end{equation*}
When $\mathfrak{q}|\mathfrak{n}$ denote $T_\mathfrak{q}$ by $U_\mathfrak{q}$ and define
\begin{equation*}
\Phi|_{U_\mathfrak{q}}(g):=\sum_{u\;\textrm{mod}\;\mathfrak{q}}\Phi\left(g\bonitaa{\pi_{\mathfrak{q}}}{u}{0}{1}\right).   
\end{equation*}

Let $\lambda_\mathfrak{p}$ denote the $T_\mathfrak{p}$ eigenvalue of $\Phi$, and let $\alpha_{\mathfrak{p}}$ and $\beta_{\mathfrak{p}}$ denote the roots of the Hecke polynomial $X^2-\lambda_{\mathfrak{p}} X+N(\mathfrak{p})^{k+1}$. Define the $\mathfrak{p}$-stabilisations of $\Phi$ to be
\begin{equation*}
\Phi^{\alpha_{\mathfrak{p}}}(g):=\Phi(g)-\alpha_{\mathfrak{p}}^{-1}\Phi\left(g\bonitaa{1}{0}{0}{\pi_\mathfrak{p}}\right),\;\;\Phi^{\beta_{\mathfrak{p}}}(g):=\Phi(g)-\beta_{\mathfrak{p}}^{-1}\Phi\left(g\bonitaa{1}{0}{0}{\pi_\mathfrak{p}}\right).
\end{equation*}
Then $\Phi^{\alpha_{\mathfrak{p}}}$ and $\Phi^{\beta_{\mathfrak{p}}}$ are eigenforms of level $\Omega_0(\mathfrak{p}\mathfrak{n})$ and $U_{\mathfrak{p}}$-eigenvalues $\alpha_{\mathfrak{p}}$ and $\beta_{\mathfrak{p}}$, in fact, for example for $\Phi^{\alpha_{\mathfrak{p}}}$ we have
\begin{align*}
\Phi^{\alpha_{\mathfrak{p}}}|_{U_\pprime}(g)&=\Phi|_{U_\pprime}(g)-\alpha_{\mathfrak{p}}^{-1}\Phi|_{U_\pprime}\left(g\bonitaa{1}{0}{0}{\pi_\mathfrak{p}}\right)\\
&=\lambda_\pprime\Phi(g) - \Phi\left(g\bonitaa{1}{0}{0}{\pi_{\mathfrak{\pprime}}}\right)-\alpha_{\mathfrak{p}}^{-1}N(\pprime)^{k+1}\Phi(g)\\
&=(\lambda_\pprime-\beta_\pprime)\Phi(g) - \Phi\left(g\bonitaa{1}{0}{0}{\pi_{\mathfrak{\pprime}}}\right)\\
&=\alpha_\pprime\Phi(g) - \Phi\left(g\bonitaa{1}{0}{0}{\pi_{\mathfrak{\pprime}}}\right)=\alpha_{\mathfrak{p}}\Phi^{\alpha_{\mathfrak{p}}}(g),
\end{align*}
where in second equality the first term comes from the definition of $T_\pprime$ as follows
\begin{equation*}
\lambda_\pprime\Phi(g)=\Phi|_{T_\pprime}(g)=\Phi|_{U_\pprime}(g) + \Phi\left(g\bonitaa{1}{0}{0}{\pi_{\mathfrak{\pprime}}}\right)\;\; \mathrm{then} \;\; \Phi|_{U_\pprime}(g)=\lambda_\pprime\Phi(g) - \Phi\left(g\bonitaa{1}{0}{0}{\pi_{\mathfrak{\pprime}}}\right)
\end{equation*}
and the second term comes from 
\begin{align*}
\Phi|_{U_\pprime}\left(g\bonitaa{1}{0}{0}{\pi_{\mathfrak{\pprime}}}\right)&=\sum_{u\;\textrm{mod}\;\pprime}\Phi\left(g\bonitaa{\pi_{\mathfrak{\pprime}}}{u}{0}{1}\bonitaa{1}{0}{0}{\pi_{\mathfrak{\pprime}}}\right)=\sum_{u\;\textrm{mod}\;\pprime}\Phi\left(g\bonitaa{1}{u}{0}{1}\bonitaa{\pi_{\mathfrak{\pprime}}}{0}{0}{\pi_{\mathfrak{\pprime}}}\right)\\
&=\varphi_\pprime(\pi_\pprime)\sum_{u\;\textrm{mod}\;\pprime}\Phi\left(g\bonitaa{1}{u}{0}{1}\right)=N(\pprime)^k\sum_{u\;\textrm{mod}\;\pprime}\Phi(g)=N(\pprime)^{k+1}\Phi(g),
\end{align*}
where third equality uses property (ii) in Definition \ref{T1.4} with the character $\varphi_\mathfrak{p}$. In forth equality we use invariance of $\Phi$ by $\Omega_0(\mathfrak{n})$ and $\varphi_\infty(\pi_{\mathfrak{p}})\varphi_\mathfrak{p}(\pi_{\mathfrak{p}})=1$ then $\varphi_\mathfrak{p}(\pi_{\mathfrak{p}})=\varphi_\infty(\pi_{\mathfrak{p}})^{-1}=|\pi_{\mathfrak{p}}|^{2k}=N(\pprime)^k$, with $|\pi_{\mathfrak{p}}|$ the archimedean norm of $\pi_{\mathfrak{p}}$ as an element of $K$.\\\hfill

Now we will find the descent of the $\pprime$-stabilisation $\Phi^{\alpha_{\mathfrak{p}}}$ to $\mathcal{H}_3$:\\\hfill

Taking $g_{\infty}=\left(\bonitaa{t}{z}{0}{1},...\bonitaa{1}{0}{0}{1},...\right)\in\mathrm{GL}_2(\mathbb{A}_K)$ with $t\in\mathbb{R}_{>0}$ and $z\in\mathbb{C}$ we write
\begin{equation*}
\Phi^{\alpha_{\mathfrak{p}}}(g_\infty)=\Phi(g_\infty)-\alpha_{\mathfrak{p}}^{-1}\Phi\left(\bonitaa{t}{z}{0}{1},...,\bonitaa{1}{0}{0}{\pi_\mathfrak{p}}...\right).   
\end{equation*}
Note that
\begin{align*}
\Phi\left(\bonitaa{t}{z}{0}{1},...,\bonitaa{1}{0}{0}{\pi_\mathfrak{p}}...\right)&=\Phi\left(\bonitaa{\pi_\mathfrak{p}}{0}{0}{1}\bonitaa{t}{z}{0}{1},...,\bonitaa{\pi_\mathfrak{p}}{0}{0}{1},...,\bonitaa{\pi_\mathfrak{p}}{0}{0}{\pi_\mathfrak{p}},...\right)\\
&=\Phi\left(\bonitaa{\pi_\mathfrak{p}}{0}{0}{1}\bonitaa{t}{z}{0}{1},...,\bonitaa{1}{0}{0}{1},...,\bonitaa{\pi_\mathfrak{p}}{0}{0}{\pi_\mathfrak{p}},...\right)\\
&=\varphi_\mathfrak{p}(\pi_\mathfrak{p})\Phi\left(\bonitaa{\pi_\mathfrak{p}}{0}{0}{1}\bonitaa{t}{z}{0}{1},...,\bonitaa{1}{0}{0}{1}...,\bonitaa{1}{0}{0}{1},...\right)\\
&=|\pi_{\mathfrak{p}}|^{2k}F\left(\bonitaa{\pi_\mathfrak{p}}{0}{0}{1}\bonitaa{t}{z}{0}{1}\right)=t|\pi_{\mathfrak{p}}|^{2k}\mathcal{F}|_{\bonitaa{\pi_{\mathfrak{p}}}{0}{0}{1}}(z,t)
\end{align*}
Where first equality follows by the left $\mathrm{GL}_2(K)$ invariance of $\Phi$ (Property (i) in Definition \ref{T1.4}) with the matrix $\bonitaa{\pi_\mathfrak{p}}{0}{0}{1}$. In second equality we use right $\Omega_0(\mathfrak{n})$ invariance of $\Phi$ for $\left(\bonitaa{\pi_\mathfrak{p}}{0}{0}{1},...,\bonitaa{\pi_\mathfrak{p}}{0}{0}{1},..\bonitaa{1}{0}{0}{1},...\right)\in \Omega_0(\mathfrak{n})$ where the place $\mathfrak{p}$ has the identity matrix. Third equality uses property (ii) in Definition \ref{T1.4} with the character $\varphi_\mathfrak{p}$. In forth equality we use the definition of $F$ and that $\varphi_\mathfrak{p}(\pi_{\mathfrak{p}})=\varphi_\infty(\pi_{\mathfrak{p}})^{-1}=|\pi_{\mathfrak{p}}|^{2k}$, with $|\pi_{\mathfrak{p}}|$ the archimedean norm of $\pi_{\mathfrak{p}}$ as an element of $K$. Finally in the last equality we use 1) in remark \ref{rem2.6}. Then
\begin{equation*}
\alpha_{\mathfrak{p}}^{-1}\Phi\left(\bonitaa{t}{z}{0}{1},...,\bonitaa{1}{0}{0}{\pi_\mathfrak{p}}...\right)=t|\pi_{\mathfrak{p}}|^{2k}\alpha_{\mathfrak{p}}^{-1}\mathcal{F}|_{\bonitaa{\pi_{\mathfrak{p}}}{0}{0}{1}}(z,t)=t\frac{\beta_{\mathfrak{p}}}{|\pi_{\mathfrak{p}}|^{2}}\mathcal{F}|_{\bonitaa{\pi_{\mathfrak{p}}}{0}{0}{1}}(z,t).  
\end{equation*} \\ \hfill

Let $\mathcal{F}^{\alpha_{\mathfrak{p}}}$ denote the descent of $\Phi^{\alpha_{\mathfrak{p}}}$ to $\mathcal{H}_3$ then we have

\begin{align*}
t\mathcal{F}^{\alpha_{\mathfrak{p}}}(z,t)=F^{\alpha_{\mathfrak{p}}}\left(\bonitaa{t}{z}{0}{1}\right)=\Phi^{\alpha_{\mathfrak{p}}}(g_{\infty})=t\mathcal{F}(z,t)-t\frac{\beta_{\mathfrak{p}}}{|\pi_{\mathfrak{p}}|^{2}}\mathcal{F}|_{\bonitaa{\pi_{\mathfrak{p}}}{0}{0}{1}}(z,t),
\end{align*}
obtaining
\begin{equation*}
\mathcal{F}^{\alpha_{\mathfrak{p}}}(z,t)=\mathcal{F}(z,t)-\frac{\beta_{\mathfrak{p}}}{|\pi_{\mathfrak{p}}|^{2}}\mathcal{F}|_{\bonitaa{\pi_{\mathfrak{p}}}{0}{0}{1}}(z,t).
\end{equation*}

Then, considering the descent of the $\pprime$-stabilisation of an automorphic form we define the $\pprime$-stabilisations of a Bianchi eigenform $\mathcal{F}\in S_{(k,k)}(\Gamma_0(\mathfrak{n}))$ to be
\begin{equation*}
\mathcal{F}^{\alpha_{\pprime}}(z,t):= \mathcal{F}(z,t)-\beta_\pprime \mathcal{G}(z,t), \;\;\; \mathcal{F}^{\beta_{\pprime}}(z,t):= \mathcal{F}(z,t)- \alpha_{\pprime} \mathcal{G}(z, t)
\end{equation*}
where 
\begin{equation*}
\mathcal{G}(z,t)= |\pi_\pprime|^{-2}\mathcal{F}|_{\bonitaa{\pi_{\mathfrak{p}}}{0}{0}{1}}(z,t).  
\end{equation*}

The $p$-stabilisations $\mathcal{F}^{\alpha_{\mathfrak{p}}}$ and $\mathcal{F}^{\beta_{\mathfrak{p}}}$ are Bianchi eigenforms of level $\Gamma_0(\pprime\mathfrak{n})$ and 
$U_\pprime$ eigenvalues $\alpha_\pprime$ and $\beta_\pprime$.

\subsection{$L$-function of a $p$-stabilisation}\hfill\\

The $L$-function of a Bianchi modular form $\mathcal{F}$ of level coprime to a prime $\pprime$ and the $L$-function of a $\pprime$-stabilisation $\mathcal{F}^{\alpha_\pprime}$ are related, in fact, if we define for a Hecke character $\chi$ of conductor $\mathfrak{f}$ and $\epsilon\in\mathbb{C}^\times$ the factor
\begin{equation}
Z_\mathfrak{p}^{\epsilon}(\chi) := \begin{cases} 1-\epsilon^{-1}\chi(\mathfrak{p})^{-1} & : \mathfrak{p}\nmid\mathfrak{f}, \\ 1 & :\mbox{otherwise}, \end{cases}
\label{e3.2}
\end{equation}
we have the following:

\begin{lemma}
Let $\psi$ be a Hecke character with conductor $\mathfrak{f}$. We have for $\epsilon\in\{\alpha_\pprime,\beta_\pprime\}$
\begin{equation*}
\Lambda(\mathcal{F}^{\epsilon},\psi)=
Z_\pprime^{\epsilon}(\psi^{-1}|\cdot|_{\mathbb{A}_K}^{k})\Lambda(\mathcal{F},\psi),
\label{l11.1}
\end{equation*}
\end{lemma}

\begin{proof}
(i) We first compute $\mathcal{G}_n$: note that if $\gamma=\smallmatrixx{\pi_\pprime}{0}{0}{1}$ then $\gamma\cdot(z,t)=(\pi_\pprime z,|\pi_\pprime|t)$ and analogously to Lemma \ref{T1.19} we have
\begin{equation*}
\rho_{2k+2}^{-1}\left(J\left(\frac{\gamma}{\sqrt{det(\gamma)}} ;(0,t)\right)\right)\matrixxx{X}{Y}^{2k+2}=\matrixxx{\pi_\pprime^{1/2}X}{\overline{\pi_\pprime}^{1/2}Y}^{2k+2},\;\;\mathrm{then}   
\end{equation*}
\begin{equation*}
\mathcal{G}(z,t)\matrixxx{X}{Y}^{2k+2}=|\pi_\pprime|^{-2}\mathcal{F}|_{\bonitaa{\pi_{\mathfrak{p}}}{0}{0}{1}}(z,t)\matrixxx{X}{Y}^{2k+2}=|\pi_\pprime|^{-k-2}\mathcal{F}(\pi_\pprime z,|\pi_\pprime|t)\matrixxx{\pi_\pprime^{1/2}X}{\overline{\pi_\pprime}^{1/2}Y}^{2k+2}.  
\end{equation*}
Then computing the $n$-th component we have
\begin{equation}
\mathcal{G}_n(z,t)=|\pi_{\pprime}|^{-1}\mathcal{F}_n(\pi_\pprime z,|\pi_\pprime|t)\left(\frac{\pi_\pprime}{|\pi_\pprime|}\right)^{k+1-n}.
\label{e3.3}
\end{equation}
(ii) Now, we relate $\Lambda(\mathcal{G},\psi)$ and $\Lambda(\mathcal{F},\psi)$ when $\pprime\nmid \mathfrak{f}$ ($\Lambda(\mathcal{G},\psi)=0$ if $\pprime|\mathfrak{f}$ ): by Theorem \ref{T8.2} and (\ref{e3.3}) if $\psi$ has infinity type $0 \leqslant (q,r) \leqslant (k,k)$ we have
\begin{align*}
\Lambda(\mathcal{G},\psi)&=(*)\sum_{b\in(\mathcal{O}_K/f)^\times}\psi_{\mathfrak{f}}(b/f)\int_{0}^{\infty}t^{q+r} \mathcal{G}_{k+q-r+1}(b/f,t)dt\\
&=(*)\sum_{b\in(\mathcal{O}_K/f)^\times}\psi_{\mathfrak{f}}(b/f)\int_{0}^{\infty}t^{q+r}\left[ |\pi_{\pprime}|^{-1}\mathcal{F}_{k+q-r+1}(\pi_\pprime b/f,|\pi_\pprime|t)\left(\frac{\pi_\pprime}{|\pi_\pprime|}\right)^{-q+r} \right]dt
\end{align*}
where $(*)=\frac{(-1)^{q+1}4}{Dw\tau(\psi^{-1})\matrixxx{2k+2}{k+q-r+1}}$.\\ \hfill

Changing variable $t\rightarrow |\pi_\pprime|^{-1}t$ we obtain
\begin{align*}
\Lambda(\mathcal{G},\psi)&=\psi_\infty(\pi_\pprime)^{-1}N(\pprime)^{-1}(*)\sum_{b\in(\mathcal{O}_K/f)^\times}\psi_{\mathfrak{f}}(b/f)\int_{0}^{\infty}t^{q+r} \mathcal{F}_{k+q-r+1}(\pi_\pprime b/f,t)dt\\
&=\psi_\mathfrak{f}(\pi_\pprime)^{-1}\psi_\infty(\pi_\pprime)^{-1}N(\pprime)^{-1}(*)\sum_{b\in(\mathcal{O}_K/f)^\times}\psi_{\mathfrak{f}}(b/f)\int_{0}^{\infty}t^{q+r} \mathcal{F}_{k+q-r+1}(b/f,t)dt\\
&=\psi(\pprime)N(\pprime)^{-1}\Lambda(\mathcal{F},\psi),
\end{align*}
where in second equality we change $\pi_\pprime b\rightarrow b$ since $\pi_\pprime\nmid f$ and in last equality we use $\psi_\mathfrak{f}(\pi_\pprime)^{-1}\psi_\infty(\pi_\pprime)^{-1}=\psi_\pprime(\pi_\pprime)=\psi(\pprime)$.
Finally we obtain
\begin{equation*}
\Lambda(\mathcal{F}^{\alpha_{\pprime}},\psi)=\Lambda(\mathcal{F},\psi)-\frac{N(\pprime)^{k+1}}{\alpha_\pprime}\Lambda(\mathcal{G},\psi)=
\begin{cases*}
\left(1-\frac{\psi(\pprime)N(\pprime)^{k}}{\alpha_\pprime}\right)\Lambda(\mathcal{F},\psi) & if $\pprime\nmid\mathfrak{f}$ \\
\Lambda(\mathcal{F},\psi) & otherwise.
\end{cases*}
\end{equation*}
Noting that
\begin{equation*}
\psi(\pprime)N(\pprime)^{k}=(\psi^{-1}(\pprime)N(\pprime)^{-k})^{-1} =(\psi^{-1}(\pprime)|x_{\pprime}|_{\mathbb{A}_K}^{k})^{-1}.
\end{equation*}
Where $x_\pprime$ is the idele associated to $\pprime$.
We have 
\begin{equation*}
\Lambda(\mathcal{F}^{\alpha_{\pprime}},\psi)=
Z_\pprime^{\alpha_{\pprime}}(\psi^{-1}|\cdot|_{\mathbb{A}_K}^{k})\Lambda(\mathcal{F},\psi)    
\end{equation*}
obtaining the result for the $\Lambda$-function of $\mathcal{F}^{\alpha_{\pprime}}$; for $\mathcal{F}^{\beta_{\pprime}}$ is analogous.
\end{proof}

\begin{remark}
Depending of the behaviour of $p$, there exist four or two $p$-stabilisations $\mathcal{F}_p$. Note that the level of $\mathcal{F}_p$ is $\Gamma_0(\mathfrak{n}\pprime\overline{\pprime})$ if $p$ splits as $\pprime\overline{\pprime}$, and $\Gamma_0(\mathfrak{n}\pprime)$ if $p$ ramifies as $\pprime^2$ or remains inert as $\pprime$.
\label{R13.3}
\end{remark}

\section{Modular symbols}

In this section we introduce Bianchi modular symbols. These are algebraic analogues of Bianchi modular forms that are easier to study $p$-adically.

Let $\Delta_0:=\mathrm{Div}^0(\mathbb{P}^1(K))$ denote the space of ``paths between cusps'' in $\mathcal{H}_3$, and let $V$ be any right $\mathrm{SL}_2(K)$-module. For a subgroup $\Gamma\subset\mathrm{SL}_2(K)$, denote the space of \textit{$V$-valued modular symbols for $\Gamma$} to be the space
\begin{equation*}
\mathrm{Symb}_{\Gamma}(V):=\mathrm{Hom}_{\Gamma}(\Delta_0,V)   
\end{equation*}
of functions satisfying the $\Gamma$-invariance property that $(\phi|\gamma)(D):=(\phi|\gamma D)|\gamma=\phi(D)$
where $\Gamma$ acts on the cusps by $\smallmatrixx{a}{b}{c}{d}\cdot r=\frac{ar+b}{cr+d}$.

For a ring $R$ recall the definition of $V_k(R)$ on section $(2)$, then define $V_{k,k}(R):=V_k(R)\otimes_R V_k(R)$. 

Note that we can identify $V_{k,k}(\mathbb{C})$ with the space of polynomials that are homogeneous of degree $k$ in two variables $X,Y$ and homogeneous of degree $k$ in two further variables $\overline{X}, \overline{Y}$. 

This space has a natural left action of $\mathrm{GL}_2(\mathbb{C})^2$ induced by the action of $\mathrm{GL}_2(\mathbb{C})$ on each factor by 

\begin{equation*}
\gamma\cdot P\left[\matrixxx{X}{Y},\matrixxx{\overline{X}}{\overline{Y}}\right]= P\left[\matrixxx{dX+bY}{cX+aY},\matrixxx{\overline{d}\overline{X}+\overline{b}\overline{Y}}{\overline{c}\overline{X}+\overline{a}\overline{Y}}\right],\;\;\;\; \gamma=\matrixx{a}{b}{c}{d}.    
\end{equation*}

\begin{remark}
This induces a right action on the dual space $V_{k,k}^*(\mathbb{C}):=\mathrm{Hom}(V_{k,k}(\mathbb{C}),\mathbb{C})$, also note that, in particular we obtain an action of $\Gamma_0(\mathfrak{n})$ on $V_{k,k}^*(\mathbb{C})$. 
\end{remark}

\begin{definition}
The space of \textit{Bianchi modular symbols of parallel weight $(k,k)$ and level $\Gamma_0(\mathfrak{n})$} is defined to be the space
\begin{equation*}
\mathrm{Symb}_{\Gamma_0(\mathfrak{n})}(V_{k,k}^*(\mathbb{C})):=\mathrm{Hom}_{\Gamma_0(\mathfrak{n})}(\Delta_0,V_{k,k}^*(\mathbb{C})).
\end{equation*}
\end{definition}

We can also define Hecke operators on the space of modular symbols, as in Section \ref{section2.4}, where the Hecke operators allow us to endow the space of Bianchi modular forms with additional structure. 

Let $\mathfrak{q}\nmid\mathfrak{n}$ be a prime ideal of $\mathcal{O}_K$ generated by the fixed uniformiser $\pi_\mathfrak{q}$. Then for $\phi \in \mathrm{Symb}_{\Gamma_0(\mathfrak{n})}(V_{k,k}^*(\mathbb{C}))$ we define the Hecke operator 

\begin{equation*}
\phi\mapsto (\phi|_{T_{\mathfrak{q}}}):=\sum_{b\in(\mathcal{O}_K/\mathfrak{q})^\times} \phi|_{\bonitaa{1}{b}{0}{\pi_{\mathfrak{q}}}} +\phi|_{\bonitaa{\pi_{\mathfrak{q}}}{0}{0}{1}}.
\end{equation*}
When $\mathfrak{q}|\mathfrak{n}$ we denote $T_\mathfrak{q}$ by $U_\mathfrak{q}$ and 
\begin{equation*}
(\phi|_{U_{\mathfrak{q}}}):=\sum_{b\in(\mathcal{O}_K/\mathfrak{q})^\times} \phi|_{\bonitaa{1}{b}{0}{\pi_{\mathfrak{q}}}}.
\end{equation*}

We can define Hecke operators over Bianchi modular symbols for each ideal $I$ of $K$ in the analogous way as for Bianchi modular forms, then the Hecke algebra acts on $\mathrm{Symb}_{\Gamma_0(\mathfrak{n})}(V_{k,k}^*(\mathbb{C}))$ and by \cite[\S3, \S8]{hida1994critical} and \cite{harder1987eisenstein}, we have:

\begin{proposition}
There is a Hecke-equivariant injection 
\begin{equation*}
S_{(k,k)}(\Gamma_0(\mathfrak{n})) \hookrightarrow \mathrm{Symb}_{\Gamma_0(\mathfrak{n})}(V_{k,k}^*(\mathbb{C})), \;\; \mathcal{F} \mapsto \phi_{\mathcal{F}}.    
\end{equation*}
\end{proposition}

Let $\mathcal{X}^{k-q}\mathcal{Y}^q\overline{\mathcal{X}}^{k-r}\overline{\mathcal{Y}}^r\in V_{k,k}^*(\mathbb{C})$ be the dual of $X^{k-q}Y^q\overline{X}^{k-r}\overline{Y}^r$. We can explicitly describe the modular symbol attached to $\mathcal{F}$ at generating divisors as
\begin{equation*}
\phi_\mathcal{F}(\{a\}-\{\infty\})= \sum_{q,r=0}^{k}c_{q,r}(a)(\mathcal{Y}-a\mathcal{X})^{k-q}\mathcal{X}^q(\overline{\mathcal{Y}}-\overline{a}\overline{\mathcal{X}})^{k-r}\overline{\mathcal{X}}^r, 
\label{p5.1}
\end{equation*}
for $a\in K$, where $c_{q,r}(a)$ is defined in Theorem \ref{T8.2}, (see \cite[Prop.2.9]{chris2017}).

\begin{remark} 
The coefficient $c_{q,r}(a)$ establishes a link between values of a modular symbol and critical $L$-values of the Bianchi modular form, this link is the key to the interpolation property satisfied by the $p$-adic $L$-function.
\end{remark}

\begin{proposition}
Let $\mathcal{F}$ be a cuspidal Bianchi modular form in $S_{(k,k)}(\Gamma_0(\mathfrak{n}))$, the symbol $\phi'_\mathcal{F}:=\phi_\mathcal{F}/\Omega_\mathcal{F}$, for $\Omega_\mathcal{F}$ as in (\ref{e7.1}), takes values in $V_{k,k}^*(E)$ for some number field $E$.
\label{P7.5}
\end{proposition}
\begin{proof}
See \cite[Prop. 2.12]{chris2017}.
\end{proof}
\section{Overconvergent modular symbols}

Part (ii) of Proposition \ref{P7.5} allows us to see the modular symbol $\phi_\mathcal{F}$ as having values in $V_{k,k}^*(L)$ for a sufficiently large $p$-adic field $L$. For suitable level groups, one can then replace this space of polynomials with a space of $p$-adic distributions and obtain the so called \textit{overconvergent modular symbols}.

\begin{definition} 
Let $\mathcal{A}(L)$ denote the space of locally analytic functions on $\tensorspace$ defined over $L$. We equip this space with a weight $(k,k)$-action of the semigroup
\begin{equation*}
\Sigma_0(p):=\left\{\bonitaa{a}{b}{c}{d}\in M_2(\tensorspace): p|c, a\in(\tensorspace)^\times, ad-bc\nequal0 \right\}    
\end{equation*}
by setting 
\begin{equation*}
\gamma \cdot \zeta(z)= (a+cz)^k \zeta\left(\frac{b+dz}{a+cz}\right).  
\end{equation*}

Denote $\mathcal{A}_{k,k}(L)$ the space $\mathcal{A}(L)$ equipped with the action above.
Let $\mathcal{D}_{k,k}(L):=\mathrm{Hom}_{\mathrm{cts}}(\mathcal{A}_{k,k}(L),L)$ denote the space of locally analytic distributions on $\tensorspace$ defined over $L$, equipped with a weight $(k,k)$ right action of $\Sigma_0(p)$ given by $\mu|\gamma(\zeta)=\mu(\gamma\cdot\zeta)$.
\label{D10.1}
\end{definition}

\begin{remark}
When $p$ ramifies as $\pprime^2$ in $K$, we can consider instead $\Sigma_0(p)$, the larger group coming from condition $c\in\pi_\pprime\tensorspace$ for $\pi_\pprime$ the fixed uniformiser at $\pprime$.
\end{remark}

For $\Gamma\subset\Sigma_0(p)$, define the space of \textit{overconvergent modular symbols of weight $(k,k)$ and level $\Gamma$} to be $\mathrm{Symb}_\Gamma(\mathcal{D}_{k,k}(L))$.

There is a natural map $\mathcal{D}_{k,k}(L)\rightarrow V_{k,k}^
*(L)$ given by dualising the inclusion of $V_{k,k}(L)$ into $\mathcal{A}_{k,k}(L)$. When $(p)|\mathfrak{n}$ for $p$ split in $K$ or when $\pprime|\mathfrak{n}$ for $p$ inert or ramified this map induces a \textit{specialisation map}
\begin{equation*}
\rho: \mathrm{Symb}_{\Gamma_0(\mathfrak{n})}(\mathcal{D}_{k,k}(L)) \longrightarrow \mathrm{Symb}_{\Gamma_0(\mathfrak{n})}(V_{k,k}^*(L)),   
\end{equation*}
noting that the source is well-defined since $\Gamma_0(\mathfrak{n})\subset\Sigma_0(p)$.

\begin{theorem}
(Control theorem, \cite[Cors. 5.9, 6.13]{chris2017}). For each prime $\pprime$ above $p$, let $\lambda_{\pprime}\in L^\times$. If $v(\lambda_{\pprime})<(k+1)/e_{\pprime}$ for all $\pprime|p$, then the restriction of the specialisation map 
\begin{equation*}
\rho: \mathrm{Symb}_{\Gamma_0(\mathfrak{n})}(\mathcal{D}_{k,k}(L))^{\{U_{\pprime}=\lambda_{\pprime}:\pprime|p\}} \xrightarrow{\sim} \mathrm{Symb}_{\Gamma_0(\mathfrak{n})}(V_{k,k}^*(L))^{\{U_{\pprime}=\lambda_{\pprime}:\pprime|p\}}    
\end{equation*}
to the simultaneous $\lambda_{\pprime}$-eigenspaces of the $U_{\pprime}$ operators is an isomorphism. Here recall that $e_{\pprime}$ is the ramification index of $\pprime|p$.
\end{theorem}

\begin{definition}
If $\mathcal{F}\in S_{(k,k)}(\Gamma_0(\mathfrak{n}))$ is an eigenform with eigenvalues $\lambda_\pprime$ for $\pprime|p$, we say $\mathcal{F}$ has \textit{small slope} if $v(\lambda_{\pprime})<(k+1)/e_{\pprime}$ for all $\pprime|p$. We say $\mathcal{F}$ has \textit{critical slope} if it has no small slope.
\label{D10.4}
\end{definition}

Thus if $\mathcal{F}$ has small slope, using the above control theorem, we get an associated overconvergent modular symbol $\Psi_{\mathcal{F}}\in \mathrm{Symb}_{\Gamma_0(\mathfrak{n})}(\mathcal{D}_{k,k}(L))$ by lifting the corresponding classical modular symbol.

\section{$p$-adic $L$ function}

\subsection{Constructing the $p$-adic $L$-function}\label{constructing}\hfill\\

The $p$-adic $L$-function of a Bianchi modular form $\mathcal{F}$ is defined as the locally analytic distribution $L_p(\mathcal{F},-)$ on $\mathfrak{X}(\mathrm{Cl}_K(p^{\infty}))$, the two-dimensional rigid space of $p$-adic characters on $\mathrm{Cl}_K(p^{\infty}):=K^\times\backslash \mathbb{A}_K^\times/\mathbb{C}^\times\prod_{v\nmid p}\mathcal{O}_v^\times$, that interpolates the classical $L$-values of $\mathcal{F}$ and satisfy certain growth properties. We recall its key properties.

Recall that there is a bijection between algebraic Hecke characters of conductor dividing $p^\infty$ and locally algebraic characters of $\mathrm{Cl}_K(p^\infty)$ such that if $\psi$ corresponds to $\psi_{p-\mathrm{fin}}$, both are equal when we restrict to the adeles away from the infinite place and the primes above $p$. 

If $\psi$ is an algebraic Hecke character of $K$ of conductor $\mathfrak{f}|p^\infty$ and infinity type $(q,r)$, then, fixing an isomorphism $\mathbb{C}\cong\mathbb{C}_p$,  we associate to $\psi$ a $K^\times$-invariant function
\begin{equation*}
\psi_{p-\mathrm{fin}}(x):(\mathbb{A}_K^\times)_\mathfrak{f}\longrightarrow \mathbb{C}_p,\;\; \psi_{p-\mathrm{fin}}(x):=\psi_\mathfrak{f}(x)\sigma_p^{q,r}(x),
\end{equation*}
where
\begin{equation}
\sigma_p^{q,r}(x) := \begin{cases} x_\mathfrak{p}^q x_{\overline{\mathfrak{p}}}^r & : p \ \mbox{splits as} \ \mathfrak{p}\overline{\mathfrak{p}}, \\ 
x_\mathfrak{p}^q \overline{x_{\mathfrak{p}}^r} & :p \ \mbox{inert or ramified}. \end{cases}
\label{e12.1}
\end{equation}

\begin{remark}
For $\alpha \in K^\times$, we have $\psi_{p-\mathrm{fin}}(x_{\alpha,p})=(\psi_{p-\mathrm{fin}})_{(p)}(\alpha)=\psi_{(p)}(\alpha)\alpha^q\overline{\alpha}^r$ for $(x_{\alpha,p})_{\mathfrak{q}}=\alpha$ when $\mathfrak{q}|(p)$ and $(x_{\alpha,p})_\mathfrak{q}=1$ otherwise.
\label{e13.10}
\end{remark}

We now can construct the $p$-adic $L$-function of a small slope eigenform $\mathcal{F}\in S_{(k,k)}(\Gamma_0(\mathfrak{n}))$. First, associate to $\mathcal{F}$ a classical Bianchi eigensymbol $\phi_{\mathcal{F}}$ with coefficients in a $p$-adic field $L$, and lift it to its corresponding unique overconvergent Bianchi eigensymbol $\Psi$. Define the $p$-adic $L$-function of $\mathcal{F}$ as the locally analytic distribution $L_p(\mathcal{F},-)$ on $\mathfrak{X}(\mathrm{Cl}_K(p^{\infty}))$ by
\begin{equation*}
L_p(\mathcal{F},-):=\Psi(\{0\}-\{\infty\})|_{(\tensorspace)^\times}.   
\end{equation*}
Then, $L_p(\mathcal{F},-)$ satisfies the interpolation and admissibility properties desired (see \cite[Defs. 5.10, 6.14]{chris2017}) and we obtain the main theorem (Theorem 7.4) in \cite{chris2017} for class number $1$:
 
\begin{theorem}
Let $\mathcal{F}$ be a cuspidal Bianchi modular eigenform of weight $(k,k)$ and level $\Gamma_0(\mathfrak{n})$, where $(p)|\mathfrak{n}$, with $U_\mathfrak{p}$-eigenvalues $\lambda_\mathfrak{p}$, where $v(\lambda_\mathfrak{p}) < (k+1)/e_\mathfrak{p}$ for all $\mathfrak{p}|p$. Let $\Omega_\mathcal{F}$ be a complex period as in (\ref{e7.1}). Then there exists a locally analytic distribution $L_p(\mathcal{F},-)$ on $\mathfrak{X}(\mathrm{Cl}_K(p^{\infty}))$ such that for any Hecke character of $K$ of conductor $\mathfrak{f}|(p^\infty)$ and infinity type $0 \leqslant (q,r) \leqslant (k,k)$, we have 
\begin{equation}
L_p(\mathcal{F},\psi_{p-\mathrm{fin}})=\left( \prod_{\mathfrak{p}|p} Z_{\mathfrak{p}}^{\lambda_\pprime}(\psi) \right) \left[ \frac{Dw\tau(\psi^{-1})}{(-1)^{k+q+r}2\lambda_\mathfrak{f}\Omega_\mathcal{F}} \right] \Lambda(\mathcal{F},\psi),
\end{equation}
with $Z_\mathfrak{p}^{\lambda_{\pprime}}(\psi)$ as in (\ref{e3.2}).

The distribution $L_p(\mathcal{F},-)$ is $(h_\mathfrak{p})_{\mathfrak{p}|p}$-admissible, where $h_\mathfrak{p}=v_p(\lambda_\mathfrak{p})$, and hence is unique.
\label{T13.3}
\end{theorem}

\begin{remark}
1) There is a slight error in \cite{chris2017}, where the term $Z_\mathfrak{p}^{\lambda_\pprime}$ is incorrect in the case where $\pprime\nmid\mathfrak{f}$.

2) When $p$ ramifies as $\pprime^2$, in the Theorem above it suffices $\pprime|\mathfrak{n}$ instead $(p)|\mathfrak{n}$.
\end{remark}

\subsection{Functional equation of the $p$-adic $L$-function}

\subsubsection{The small slope case}\label{subsub}\hfill\\

In this section we obtain the functional equation of the $p$-adic $L$-function of a small slope $p$-stabilisation of a cuspidal Bianchi modular eigenform. 

Let $\mathcal{F}_p$ be a Bianchi modular form obtained by successively stabilising at each different prime $\pprime$ above $p$ a newform $\mathcal{F}\in S_{(k,k)}(\Gamma_0(\mathfrak{n}))$, with $\mathfrak{n}=(\nu)$ prime to $(p)$. Recall that $\mathcal{F}$ is an eigenform for the Fricke involution $W_\mathfrak{n}$, with $\mathcal{F}|_{W_\mathfrak{n}}=\epsilon(\mathfrak{n})\mathcal{F}$ with $\epsilon(\mathfrak{n})=\pm1$.

\begin{lemma}
For any Hecke character $\psi$ of conductor $\mathfrak{f}=(f)$ with $(f,\nu)=1$ and infinity type $0 \leqslant (q,r) \leqslant (k,k)$ we have 
\begin{equation*}
\left(\prod_{\pprime|p}Z_\pprime^{\alpha_{\pprime}}(\psi)\right)\Lambda(\mathcal{F}_p,\psi)=\varepsilon(\mathcal{F},\psi)\left(\prod_{\pprime|p}Z_\pprime^{\alpha_{\pprime}}(\psi^{-1}|\cdot|_{\mathbb{A}_K}^{k})\right)\Lambda(\mathcal{F}_p,\psi^{-1}|\cdot|_{\mathbb{A}_K}^{k}),
\end{equation*}
where $\varepsilon(\mathcal{F},\psi)=\left[\frac{-\epsilon(\mathfrak{n})|\nu|^k\tau(\psi|\cdot|_{\mathbb{A}_K}^{-k})}{\psi_\mathfrak{f}(-\nu)\psi_\infty(-\nu)\tau(\psi^{-1})}\right]$ and $\alpha_\pprime$ are the $U_\pprime$-eigenvalues of $\mathcal{F}_p$ for each $\pprime|p$. 
\label{L9.3}
\end{lemma}

\begin{proof}
By Theorem \ref{T16.1} we have
\begin{equation}
\Lambda(\mathcal{F},\psi)=\varepsilon(\mathcal{F},\psi)\Lambda(\mathcal{F},\psi^{-1}|\cdot|_{\mathbb{A}_K}^{-k})  
\label{e14.2}
\end{equation}
Now let $\pprime$ be a prime over $p$, if we define $\mathcal{F}^{\alpha_{\pprime}}$ as a $\pprime-$stabilisation of $\mathcal{F}$, we obtain by Lemma \ref{l11.1} the following relations between the $\Lambda$-function of $\mathcal{F}^{\alpha_{\pprime}}$ and $\mathcal{F}$
\begin{equation}
\Lambda(\mathcal{F}^{\alpha_{\pprime}},\psi)=
Z_\pprime^{\alpha_{\pprime}}(\psi^{-1}|\cdot|_{\mathbb{A}_K}^{k})\Lambda(\mathcal{F},\psi),
\label{e14.3}
\end{equation}
\begin{equation}
\Lambda(\mathcal{F}^{\alpha_{\pprime}},\psi^{-1}|\cdot|_{\mathbb{A}_K}^{k})=Z_\pprime^{\alpha_{\pprime}}(\psi)\Lambda(\mathcal{F},\psi^{-1}|\cdot|_{\mathbb{A}_K}^{k}).
\label{e14.4}
\end{equation}
Putting (\ref{e14.3}), (\ref{e14.2}) and (\ref{e14.4}) together
\begin{align*}
Z_\pprime^{\alpha_{\pprime}}(\psi)\Lambda(\mathcal{F}^{\alpha_{\pprime}},\psi)&=Z_\pprime^{\alpha_{\pprime}}(\psi)Z_\pprime^{\alpha_{\pprime}}(\psi^{-1}|\cdot|_{\mathbb{A}_K}^{k})\Lambda(\mathcal{F},\psi)\\
&=Z_\pprime^{\alpha_{\pprime}}(\psi)Z_\pprime^{\alpha_{\pprime}}(\psi^{-1}|\cdot|_{\mathbb{A}_K}^{k})\varepsilon(\mathcal{F},\psi)\Lambda(\mathcal{F},\psi^{-1}|\cdot|_{\mathbb{A}_K}^{k})\\
&=\varepsilon(\mathcal{F},\psi)Z_\pprime^{\alpha_{\pprime}}(\psi^{-1}|\cdot|_{\mathbb{A}_K}^{k})\Lambda(\mathcal{F}^{\alpha_{\pprime}},\psi^{-1}|\cdot|_{\mathbb{A}_K}^{k}).
\end{align*}
Note that if $p$ is inert or ramified we are done and $\mathcal{F}_p=\mathcal{F}^{\alpha_{\pprime}}$. If $p$ split we have to do one more stabilisation, let $\overline{\pprime}$ be the other prime above $p$.

If we define $\mathcal{F}^{\alpha_{\pprime},\alpha_{\overline{\pprime}}}$ as the $\overline{\pprime}-$stabilisation of $\mathcal{F}^{\alpha_{\pprime}}$ and doing the same process above, we obtain
\begin{equation*}
\left(\prod_{\pprime|p}Z_\pprime^{\alpha_{\pprime}}(\psi)\right)\Lambda(\mathcal{F}^{\alpha_{\pprime},\alpha_{\overline{\pprime}}},\psi)=\varepsilon(\mathcal{F},\psi)\left(\prod_{\pprime|p}Z_\pprime^{\alpha_{\pprime}}(\psi^{-1}|\cdot|_{\mathbb{A}_K}^{k})\right)\Lambda(\mathcal{F}^{\alpha_{\pprime},\alpha_{\overline{\pprime}}},\psi^{-1}|\cdot|_{\mathbb{A}_K}^{k}).
\end{equation*}
Putting $\mathcal{F}_p=\mathcal{F}^{\alpha_{\pprime},\alpha_{\overline{\pprime}}}$ when $p$ split we obtain the result.
\end{proof}

\begin{proposition}
If $\mathcal{F}_p$ has small slope, then for any Hecke character $\psi$ of conductor $\mathfrak{f}|(p^\infty)$ with $\mathfrak{f}=(f)$ and infinity type $0 \leqslant (q,r) \leqslant (k,k)$, the  distribution $L_p(\mathcal{F}_p,-)$ satisfies
\begin{equation*}
L_p(\mathcal{F}_p,\psi_{p-\mathrm{fin}})=-\epsilon(\mathfrak{n})\mathrm{N}(\mathfrak{n})^{k/2}\psi^{-1}_{p-\mathrm{fin}}(x_{-\nu,p})L_p(\mathcal{F}_p,\psi_{p-\mathrm{fin}}^{-1}\sigma_p^{k,k}),
\end{equation*}
where $x_{-\nu,p}$ is the idele associated to $-\nu$ defined in Remark \ref{e13.10} and $\sigma_p^{k,k}$ as in equation (\ref{e12.1}).
\label{p14.1}
\end{proposition} 

\begin{proof}

By Theorem \ref{T13.3} we have the following interpolations
\begin{equation}
L_p(\mathcal{F}_p,\psi_{p-\mathrm{fin}})=\left( \prod_{\mathfrak{p}|p} Z_\mathfrak{p}^{\alpha_{\pprime}}(\psi) \right) \left[ \frac{Dw\tau(\psi^{-1})}{(-1)^{k+q+r}2\lambda_\mathfrak{f}\Omega_\mathcal{F}} \right] \Lambda(\mathcal{F}_p,\psi).
\label{e6.6}
\end{equation}
\begin{equation}
L_p(\mathcal{F}_p,(\psi^{-1}|\cdot|_{\mathbb{A}_K}^{k})_{p-\mathrm{fin}})=\left( \prod_{\mathfrak{p}|p} Z_\mathfrak{p}^{\alpha_{\pprime}}(\psi^{-1}|\cdot|_{\mathbb{A}_K}^{k}) \right) \left[ \frac{Dw\tau(\psi|\cdot|_{\mathbb{A}_K}^{-k})}{(-1)^{k+q+r}2\lambda_\mathfrak{f}\Omega_\mathcal{F}} \right] \Lambda(\mathcal{F}_p,\psi^{-1}|\cdot|_{\mathbb{A}_K}^{k}).
\label{e8}
\end{equation}
By (\ref{e6.6}), Lemma \ref{L9.3} and (\ref{e8}) we have
\begin{align*}
L_p(\mathcal{F}_p,\psi_{p-\mathrm{fin}})&=\left( \prod_{\mathfrak{p}|p} Z_\mathfrak{p}^{\alpha_{\pprime}}(\psi) \right) \left[ \frac{Dw\tau(\psi^{-1})}{(-1)^{k+q+r}2\lambda_\mathfrak{f}\Omega_\mathcal{F}} \right] \Lambda(\mathcal{F}_p,\psi)\\  
&=\left[ \frac{Dw\tau(\psi^{-1})}{(-1)^{k+q+r}2\lambda_\mathfrak{f}\Omega_\mathcal{F}} \right]\varepsilon(\mathcal{F},\psi)\left(\prod_{\pprime|p}Z_\pprime^{\alpha_{\pprime}}(\psi^{-1}|\cdot|_{\mathbb{A}_K}^{k})\right)\Lambda(\mathcal{F}_p,\psi^{-1}|\cdot|_{\mathbb{A}_K}^{k})\\
&=\left[ \frac{Dw\tau(\psi^{-1})}{(-1)^{k+q+r}2\lambda_\mathfrak{f}\Omega_\mathcal{F}} \right]\varepsilon(\mathcal{F},\psi)\left[ \frac{Dw\tau(\psi|\cdot|_{\mathbb{A}_K}^{-k})}{(-1)^{k+q+r}2\lambda_\mathfrak{f}\Omega_\mathcal{F}} \right]^{-1}L_p(\mathcal{F}_p,(\psi^{-1}|\cdot|_{\mathbb{A}_K}^{k})_{p-\mathrm{fin}})\\
&=\varepsilon(\mathcal{F},\psi)\tau(\psi^{-1})\tau(\psi|\cdot|_{\mathbb{A}_K}^{-k})^{-1}L_p(\mathcal{F}_p,(\psi^{-1}|\cdot|_{\mathbb{A}_K}^{k})_{p-\mathrm{fin}})\\
&=-\epsilon(\mathfrak{n})|\nu|^k\psi_\mathfrak{f}^{-1}(-\nu)\psi_\infty^{-1}(-\nu)L_p(\mathcal{F}_p,(\psi^{-1}|\cdot|_{\mathbb{A}_K}^{k})_{p-\mathrm{fin}}).
\end{align*}
By Remark \ref{e13.10} we have 
\begin{align*}
\psi_\mathfrak{f}^{-1}(-\nu)\psi_\infty^{-1}(-\nu)&=\psi^{-1}_{(p)}(-\nu)(-\nu)^{-q}(-\overline{\nu})^{-r}\\
&=(\psi^{-1}_{p-\mathrm{fin}})_{(p)}(-\nu)\\
&=\psi^{-1}_{p-\mathrm{fin}}(x_{-\nu,p}),
\end{align*}
and noting that for a finite idele $x$ we have 
\begin{equation*}
(\psi^{-1}|\cdot|_{\mathbb{A}_K}^{k})_{p-\mathrm{fin}}(x)=\psi_{p-\mathrm{fin}}^{-1}(x)(|\cdot|_{\mathbb{A}_K}^{k})_{p-\mathrm{fin}}(x)=
\psi_{p-\mathrm{fin}}^{-1}(x)\sigma_p^{k,k}(x)=(\psi_{p-\mathrm{fin}}^{-1}\sigma_p^{k,k})(x),
\end{equation*}
we obtain the result.
\end{proof}

\begin{theorem}
For $\mathcal{F}_p$ as above with small slope, the distribution $L_p(\mathcal{F}_p,-)$ satisfies the following functional equation
\begin{equation*}
L_p(\mathcal{F}_p,\kappa)=-\epsilon(\mathfrak{n})\mathrm{N}(\mathfrak{n})^{k/2}\kappa(x_{-\nu,p})^{-1}L_p(\mathcal{F}_p,\kappa^{-1}\sigma_p^{k,k}),
\end{equation*}
for all $\kappa\in\mathfrak{X}(\mathrm{Cl}_K(p^{\infty}))$.
\label{T14.2}
\end{theorem} 

\begin{proof}
Define a new distribution $L_p'(\mathcal{F}_p,-)$ by 
\begin{equation*}
L_p'(\mathcal{F}_p,\kappa):= L_p(\mathcal{F}_p,\kappa)+\epsilon(\mathfrak{n})\mathrm{N}(\mathfrak{n})^{k/2}\kappa(x_{-\nu,p})^{-1}L_p(\mathcal{F}_p,\kappa^{-1}\sigma_p^{k,k})  
\end{equation*}
for any $\kappa\in\mathfrak{X}(\mathrm{Cl}_K(p^{\infty}))$.

Since $v_p(\alpha_\mathfrak{p}) < (k+1)/e_\pprime$ for all $\mathfrak{p}|p$, the distribution $L_p(\mathcal{F}_p,-)$ is $(h_\mathfrak{p})_{\mathfrak{p}|p}$-admissible, where $h_\mathfrak{p}=v_p(\alpha_\mathfrak{p})$. Then $L_p'(\mathcal{F}_p,-)$ is $(h_\mathfrak{p})_{\mathfrak{p}|p}$-admissible.

In \cite{loeffler2014p} it is proved that a distribution $(h_\mathfrak{p})_{\mathfrak{p}|p}$-admissible like $L_p'(\mathcal{F}_p,-)$ is uniquely determined by its values on the $p$-adic characters $\psi_{p-\mathrm{fin}}\in\mathfrak{X}(\mathrm{Cl}_K(p^{\infty}))$ that arise from Hecke characters $\psi$ of conductor $\mathfrak{f}|(p^\infty)$ and infinity type $0 \leqslant (q,r) \leqslant (k,k)$. By Proposition \ref{p14.1} we have that $L_p'(\mathcal{F}_p,\psi_{p-\mathrm{fin}})=0$ for all $\psi_{p-\mathrm{fin}}$, then $L_p'(\mathcal{F}_p,-)=0$ and the functional equation of $L_p(\mathcal{F}_p,-)$ follows.
\end{proof}

\begin{remark}
In the case when $p$ split, the property that $L_p'(\mathcal{F}_p,-)$ is uniquely determined by its values on the $p$-adic characters $\psi_{p-\mathrm{fin}}$; is proved
in Theorem 3.11 in \cite{loeffler2014p} in the case where $v_p(\alpha_{\pprime})<1$ for $\pprime$ and $\overline{\pprime}$, which he assumes merely for simplicity. %{\color{red}Aquí hay que especificar algo sobre el resultado de Loeffler para p inerte o ramificado?.} 
For a more detailed example of the general situation in the one variable case, see \cite{colmez2010fonctions}.
\end{remark}

\textbf{Example:} Suppose $p$ splits in $K$ as $\pprime\overline{\pprime}$. Let $\mathcal{F}$ be a newform with weight $(k,k)$ and level $\mathfrak{n}$ prime to $p$ with $\lambda_\pprime=\lambda_{\overline{\pprime}}=0$. Then the Hecke polynomials at $\pprime$ and $\overline{\pprime}$ coincide, and their roots $\alpha$, $\beta$ both have $p$-adic valuation $(k+1)/2$. Assuming $\alpha\neq\beta$ there are four choices of stabilisations of level $(p)\mathfrak{n}$ and each is small slope, giving rise to four $p$-adic $L$-functions attached to $\mathcal{F}$, each one satisfying the corresponding $p$-adic functional equation of Theorem \ref{T14.2}.

\subsubsection{The critical slope case}\label{critical}\hfill\\

The construction of the $p$-adic $L$-function in \cite{chris2017} and consequently the functional equation in Theorem \ref{T14.2} depend of the small slope of the Bianchi modular form $\mathcal{F}$. In this section we generalize the functional equation of Theorem \ref{T14.2} for $\Sigma$-smooth base-change Bianchi modular forms, in particular, making no assumption about the slope. To this end we use the \textit{three-variable $p$-adic $L$-function} constructed in \cite{salazar2018} that specialises to $L_p(f_{/K},-)$, the $p$-adic $L$-function of a base-change Bianchi modular form $f_{/K}$.

We first recall briefly the definitions and construction of such $p$-adic $L$-function.

\begin{Conditions}
Let $N$ be divisible by $p$. Fix $f\in S_{k+2}(\Gamma_0(N))$ such that:
\begin{enumerate}
\item[(C1)] (finite slope eigenform) $f$ is an eigenform, and $U_pf=\lambda_pf$ with $\lambda_p\neq0$;
\item[(C2)] ($p$-stabilised newform) $f$ is new or the $p$-stabilisation of a newform $f_{\mathrm{new}}$ of level prime to $p$;
\item[(C3)] (regular) if $f$ is the $p$-stabilisation of $f_{\mathrm{new}}$, then the Hecke polynomial at $p$ of $f_{\mathrm{new}}$ has two different roots, and if $p$ is inert in $K$, $a_p(f_{\mathrm{new}})\neq0$; 
\item[(C4)] (non CM) $f$ does not have CM by $K$;
\item[(C5)] $f$ is decent (see \cite[Def. 5.5]{salazar2018});
\item[(C6)] $f_{/K}$ is $\Sigma$-smooth (see \cite[Def 5.12]{salazar2018}).
\end{enumerate}
\label{conditions}
\end{Conditions}

There exists a neighbourhood $V_\mathbb{Q}$ of $f$ in the Coleman-Mazur eigencurve such that the weight map $w$ is étale except possibly at $f$.

Up to shrinking $V_\mathbb{Q}$ there exists a unique \textit{rigid-analytic} function
\begin{equation*}
\mathcal{L}_p: V_{\mathbb{Q}}\times \mathfrak{X}(\mathrm{Cl}_K(p^{\infty})) \rightarrow L,    
\end{equation*}
for sufficiently large $L\subset\overline{\mathbb{Q}}_p$, such that for any classical point $y \in V_\mathbb{Q}(L)$ with small slope base-change $f_{y/K}$ we have $\mathcal{L}_p(y,-)=c_yL_p(f_{y/K},-)$, where $c_y\in L^\times$ is a $p$-adic period at $y$ and $L_p(f_{y/K},-)$ is the $p$-adic $L$-function of $f_{y/K}$ of Theorem \ref{T13.3}.

Note that in \cite{salazar2018}, $\mathcal{L}_p$ depends of $\phi$, a finite order Hecke character of $K$ of conductor prime to $p\mathcal{O}_K$, and is denoted by $\mathcal{L}_p^\phi$. Here we take $\phi$ to be trivial.

\begin{remark}
Shrinking $V_\mathbb{Q}$ we can suppose that there exists a Zariski-dense set $S\subset V_\mathbb{Q}$ of classical points such that for every $y\in S$ we have
\begin{itemize}
\item[(i)] $f_{y/K}$ is a successive $p$-stabilisation at each prime $\pprime$ above $p$ of a Bianchi newform of level $\Gamma_0(\mathfrak{n})$ where $\mathfrak{n}=(\nu)$ is the prime-to-$p$ part of the level of $f_{/K}$.
\item[(ii)] $f_{y/K}$ has small slope.
\item[(iii)] The weight $(k_y,k_y)$ of $f_{y/K}$ satisfies $k_y\equiv k \pmod{p-1}$.
\end{itemize}
Where condition (iii) comes since we are working in one of the $(p-1)$ discs in the weight space.
\label{R14.2}
\end{remark}

For purposes of $p$-adic variation of the weight we have to give meaning to $p$-adic exponents.

\begin{definition}
Let $\mathfrak{p}|p$ and $s\in\mathcal{O}_\mathfrak{p}$, define the function $\langle \cdot \rangle^s:=\exp(s\cdot\log_p(\langle \cdot \rangle))$ in $\mathcal{O}_\pprime^\times$, where $\log_p$ denotes the $p$-adic logarithm and $\langle\cdot\rangle$ is the projection of $\mathcal{O}_\mathfrak{p}^\times$ to $1+\mathfrak{p}^{r_\mathfrak{p}}\mathcal{O}_\mathfrak{p}$ for $r_\mathfrak{p}$ the smallest positive integer such that the usual $p$-adic exponential map converges on $\mathfrak{p}^{r_\mathfrak{p}}\mathcal{O}_\mathfrak{p}$. Define by $\langle\cdot\rangle^\textbf{s}= \prod_{\pprime|p} \langle\cdot\rangle^{s_\pprime}$ with $\textbf{s}=(s_\pprime)_{\pprime|p}\in \tensorspace \cong \prod_{\pprime|p}\mathcal{O}_\pprime$ the corresponding function in $(\tensorspace)^\times$. Let $w_{\mathrm{Tm},\pprime}: \mathcal{O}_\pprime^\times \rightarrow (\mathcal{O}_\pprime/\pprime^{r_\pprime})^\times \subset \mathcal{O}_\pprime^\times$ denote the Teichmüller character at $\pprime$, so that for $z \in \mathcal{O}_\pprime^\times$, we have $z= w_{\mathrm{Tm},\pprime}(z)\langle z \rangle$. Also let $w_\mathrm{Tm}:= \prod_{\pprime|p} w_{\mathrm{Tm},\pprime}$ be the corresponding character of $(\tensorspace)^\times$.
\label{D14.2}
\end{definition}

Recall the definition of $\sigma_p^{k,k}(x)$ in equation (\ref{e12.1}) and note that, for example, for $x\in (\mathbb{A}_K^\times)_f$ we have $\sigma_p^{k,k}(x)=[\langle x_p \rangle w_{\mathrm{Tm}}(x_p)]^k$, where $x_p=(x_\pprime)_{\pprime|p}$.
\begin{theorem}
Let $V_\mathbb{Q}$ as in Remark \ref{R14.2}, then for every $y\in V_\mathbb{Q}$ and $\kappa\in\mathfrak{X}(\mathrm{Cl}_K(p^{\infty}))$ we have
\begin{equation*}
\mathcal{L}_p(y,\kappa)=-\epsilon(\mathfrak{n})w_{\mathrm{Tm}}(\mathrm{N}(\mathfrak{n}))^{k/2}  \langle \mathrm{N}(\mathfrak{n}) \rangle^{k_y/2}\kappa(x_{-\nu,p})^{-1}\mathcal{L}_p(y,\kappa^{-1}w_{\mathrm{Tm}}^{k}\langle \cdot \rangle^{k_y}),
\end{equation*}
where $\epsilon(\mathfrak{n})=\pm1$ is the eigenvalue of $f_{y/K}$ for the Fricke involution $W_\mathfrak{n}$, $x_{-\nu,p}$ is the idele associated to $-\nu$ defined in Remark (\ref{e13.10}). 
\label{T15.2}
\end{theorem} 
\begin{proof}
Consider $S\subset V_\mathbb{Q}$ as in Remark \ref{R14.2} and note that for $y\in S$, by Theorem \ref{T14.2}, the distribution $L_p(f_{y/K},-)$ satisfies for all $\kappa\in\mathfrak{X}(\mathrm{Cl}_K(p^{\infty}))$ the following functional equation
\begin{equation*}
L_p(f_{y/K},\kappa)=-\epsilon(\mathfrak{n})\mathrm{N}(\mathfrak{n})^{k_{y}/2}\kappa(x_{-\nu,p})^{-1}L_p(f_{y/K},\kappa^{-1}\sigma_p^{k_{y},k_{y}}),
\end{equation*}
multiplying both sides by the $p$-adic period $c_y$, we have
\begin{equation*}
\mathcal{L}_p(y,\kappa)=-\epsilon(\mathfrak{n})\mathrm{N}(\mathfrak{n})^{k_y/2}\kappa(x_{-\nu,p})^{-1}\mathcal{L}_p(y,\kappa^{-1}\sigma_p^{k_y,k_y}).
\end{equation*}
Note that $N(\mathfrak{n})=w_{\mathrm{Tm}}(\mathrm{N}(\mathfrak{n}))\langle \mathrm{N}(\mathfrak{n}) \rangle$ where each factor is well defined because $\pprime\nmid\mathfrak{n}$ for all $\pprime|p$, also, since $k_y\equiv k \pmod{p-1}$ we have $w_{\mathrm{Tm}}(\mathrm{N}(\mathfrak{n}))^{k_y/2}=w_{\mathrm{Tm}}(\mathrm{N}(\mathfrak{n}))^{k/2}$, then for all $y\in S$
\begin{equation*}
\mathcal{L}_p(y,\kappa)=-\epsilon(\mathfrak{n})w_{\mathrm{Tm}}(\mathrm{N}(\mathfrak{n}))^{k/2}  \langle \mathrm{N}(\mathfrak{n}) \rangle^{k_y/2}\kappa(x_{-\nu,p})^{-1}\mathcal{L}_p(y,\kappa^{-1}w_{\mathrm{Tm}}^{k}\langle \cdot \rangle^{k_y}).
\end{equation*}
Finally, since $S$ is Zariski-dense on $V_\mathbb{Q}$, then the functional equation hold for every $y\in V_\mathbb{Q}$.
\end{proof}

\begin{corollary}
Let $\mathcal{F}$ be a $\Sigma$-smooth base-change to $K$ of a modular form satisfying Conditions \ref{conditions}, let $\mathfrak{n}=(\nu)$ be the prime-to-$p$ part of the level of $\mathcal{F}$, then for all $\kappa\in\mathfrak{X}(\mathrm{Cl}_K(p^{\infty}))$ the distribution $L_p(\mathcal{F},-)$ satisfies the following functional equation
\begin{equation*}
L_p(\mathcal{F},\kappa)=-\epsilon(\mathfrak{n})\mathrm{N}(\mathfrak{n})^{k/2}\kappa(x_{-\nu,p})^{-1}L_p(\mathcal{F},\kappa^{-1}\sigma_p^{k,k}).
\end{equation*}
\end{corollary}

\begin{proof}
Let $x$ be the classic point in the Coleman-Mazur eigencurve such that $\mathcal{F}=f_{x/K}$, then 
specialise the functional equation in Theorem \ref{T15.2} at $x$.
\end{proof}

Notice that corollary above generalises Theorem \ref{T14.2} with no non-critical assumption on the $\Sigma$-smooth base-change Bianchi modular form. In particular, we obtain the functional equation of $L_p(\mathcal{F},-)$ for a  $\Sigma$-smooth small slope base-change $\mathcal{F}$ new at $p$, this case is interesting, considering for example, that when $p$ is split in $K$, small slope is automatic. 

\bibliographystyle{amsalpha}
\bibliography{ref}

\end{document}